\documentclass[12pt]{amsart}
\usepackage{amscd}
\usepackage{amsmath}

\usepackage[all]{xy}

\usepackage{pb-diagram,pb-xy}
\usepackage{amssymb}
\usepackage{enumitem}

\numberwithin{equation}{section}

\newtheorem{thm}[equation]{Theorem}
\newtheorem{prop}[equation]{Proposition}

\newtheorem{lemma}[equation]{Lemma}
\newtheorem{cor}[equation]{Corollary}

\theoremstyle{definition}
\newtheorem{rem}[equation]{Remark}
\newtheorem{example}[equation]{Example}
\newtheorem{dfn}[equation]{Definition}
\newtheorem{notation}[equation]{Notation}

\newcommand{\onto}{\rightarrow\!\!\!\!\!\rightarrow}

\newcommand{\Gal}{\mathop{\mathrm{Gal}}}

\newcommand{\rk}{\mathop{\mathrm{rk}}}

\newcommand{\Ch}{\mathop{\mathrm{Ch}}\nolimits}

\newcommand{\mult}{\operatorname{mult}}

\newcommand{\id}{\mathrm{id}}

\newcommand{\Z}{\mathbb{Z}}
\newcommand{\F}{\mathbb{F}}

\newcommand{\Mor}{\operatorname{Mor}}
\newcommand{\Spec}{\operatorname{Spec}}
\newcommand{\End}{\operatorname{End}}
\newcommand{\Hom}{\operatorname{Hom}}

\newcommand{\Prod}{\operatornamewithlimits{\textstyle\prod}}
\newcommand{\Sum}{\operatornamewithlimits{\textstyle\sum}}
\newcommand{\Oplus}{\operatornamewithlimits{\textstyle\bigoplus}}

\newcommand{\compose}{\circ}

\newcommand{\CM}{\operatorname{CM}}
\newcommand{\AM}{\operatorname{AM}}
\newcommand{\CMe}{\operatorname{CM}_{\mathrm {eff}}}
\newcommand{\CC}{\operatorname{CC}}

\newcommand{\corr}{\rightsquigarrow}

\newcommand{\D}{D}

\renewcommand{\phi}{\varphi}

\newcommand{\Fm}{\mathbf{m}}

\DeclareMathAlphabet{\cat}{OT1}{cmss}{m}{sl}

\title
[A-upper motives]
{A-upper motives
of reductive groups}


\author
[C. De Clercq]
{Charles De Clercq}

\address
{Universit\'e Sorbonne Paris Nord\\
Intitut Galil\'ee\\
Laboratoire Analyse, G\'eom\'etrie et Applications\\
Villetaneuse, FRANCE}

\email
{declercq@math.univ-paris13.fr}
\urladdr{www.math.univ-paris13.fr/~declercq}

\author
[N. Karpenko]
{Nikita Karpenko}

\address
{Mathematical \& Statistical Sciences \\
University of Alberta \\
Edmonton
\\
CANADA}

\email
{karpenko@ualberta.ca}
\urladdr{www.ualberta.ca/~karpenko}

\author
[A. Qu\'eguiner-Mathieu]
{Anne Qu\'eguiner-Mathieu}

\address
{Universit\'e Sorbonne Paris Nord\\
Intitut Galil\'ee\\
Laboratoire Analyse, G\'eom\'etrie et Applications\\
Villetaneuse, FRANCE}

\email
{queguin@math.univ-paris13.fr}
\urladdr{www.math.univ-paris13.fr/~queguin}

\date
{5 Mar 2025}

\thanks
{This work has been accomplished when the second author was a visitor at the Institut des Hautes Etudes Scientifiques.}

\begin{document}

\begin{abstract}
Given a prime number $p$, we perform the study of Chow motives and motivic decompositions, with coefficients in $\mathbb{Z}/p\mathbb{Z}$, of projective homogeneous varieties for {\em $p'$-inner} $p$-consistent reductive algebraic groups. Assorted with the known case of $p$-inner reductive groups, our results cover all absolutely simple groups of type not $^3\!\cat{D}_4$ or $^6\!\cat{D}_4$, among other examples. First, we define the {\em A-upper motives} of such a reductive group $G$; they are indecomposable motives, naturally related to Artin motives built out of spectra of subextensions of a minimal extension over which $G$ become of inner type. With this in hand, we carry on the qualitative study of motivic decompositions for projective $G$-homogeneous varieties. Providing geometric isomorphism criteria for A-upper motives, we obtain a classification of motives of projective $G$-homogeneous varieties, by means of their {\em higher Artin-Tate traces}. We also show that the {\em higher Tits $p$-indexes} of the group $G$ determine its motivic equivalence class.
\end{abstract}

\maketitle

\tableofcontents

\section{Introduction}
\label{Introduction}

Envisioned by Alexander Grothendieck in the sixties, Chow motives provide powerful invariants to study arithmetic and geometry of smooth projective varieties over fields.
The case of projective homogeneous varieties has received a lot of attention over the years and numerous breakthroughs and solutions to classical conjectures on related algebraic objects were obtained using these new methods.
This story started within quadratic form theory.
An extensive study of motives of projective quadrics, which were essential to Voevodsky's proof of the Milnor conjecture \cite{MR2031199}, was carried out by
Alexander Vishik in \cite{MR2066515}.
It played a crucial role in the second author's proof of Hoffmann's conjecture \cite{MR1992018}, as well as advances on the Kaplansky problem \cite{Vishik-u-invariant}.
A similar study was later initiated for other types of groups and varieties, which also led to many applications, notably on the isotropy of orthogonal involutions \cite{MR3022954}, the classification of motivic decompositions for exceptional groups \cite{shells}, and the classification of algebraic groups~\cite{motequiv} and motives of projective homogeneous varieties using their higher isotropy~\cite{TateTraces}.

Most of these results are in the framework of algebraic groups of inner type, or of $p$-inner type, where motives have coefficients in $\Z/p\Z$. This means the $\ast$-action of the absolute Galois group of the base field on the associated Dynkin diagram is trivial, or becomes trivial under a $p$-power field extension.
In this work, we initiate the study of motives and motivic decompositions for projective homogeneous varieties under arbitrary reductive groups.

One of the main tools used by Vishik to describe the motivic structure of projective quadrics are the motives of \v{C}ech simplicial schemes associated to orthogonal Grassmannians and living in a Voevodsky motivic category.
Working with the Chow motives with coefficients in $\mathbb{Z}/p\mathbb{Z}$, the second author extended some of his results to projective $G$-homogeneous varieties, for $G$ an arbitrary reductive group of $p$-inner type. Indecomposable summands are then obtained using the notion of {\em upper motive}~\cite{upper}, \cite{outer}.
In this article, we stick to Chow motives with coefficients in $\Z/p\Z$ and work in the framework of \emph{$p'$-inner} groups, that is groups which become of inner type over a prime-to-$p$ extension of $F$. All absolutely simple groups of type not $^6\!\cat{D}_4$  are either $p$-inner or $p'$-inner for any given prime number $p$.
In this new setting, we introduce a notion of {\em A-upper motive} (see Definition \ref{AupperG}). These indecomposable motives are naturally associated with the Artin motives given by the spectra of subfields of a minimal field extension over which $G$ becomes of inner type. For groups of inner type,
the notion of A-upper motive coincides with the classical notion of upper motive.
With this in hand, we obtain various results on motivic decomposition and classification of motives of projective homogeneous varieties, as well as classification of algebraic groups.

More specifically, the contents of the article are as follows. Sections 2 and 3 recall some known facts on Chow motives and Artin motives. In \S\,4, we introduce the functor $\Fm$, from the category of effective Chow motives to the category of Artin motives. Based on this, A-upper motives are introduced in \S\,5. Theorem~\ref{Aupper.thm} is an important technical result, which provides conditions under which the A-upper motives of a variety are classified by their image through the functor $\Fm$.
The qualitative study of motivic decompositions of projective homogeneous varieties for $p'$-inner reductive groups is done in \S\,6. Theorem \ref{main} states that up to Tate shifts, indecomposable motivic summands of projective $G$-homogeneous varieties for such a group $G$ are the A-upper motives of $G$. In the remaining part of the paper, we assume in addition that the groups are $p$-consistent, see Definition~\ref{p-consistent}. Under this condition, using the technical tools developed in the previous two sections, we provide some geometric isomorphism criteria for A-upper motives, see~Theorem \ref{crit} and Corollary \ref{critR}. With this in hand, we obtain the complete classification of motives of these projective $G$-homogeneous varieties, through their {\em higher Artin-Tate traces} (Theorem \ref{isomcrit}). This generalizes the main result of \cite{TateTraces}, in an optimal way (see Remark \ref{counterexample}).
Finally, in \S\ref{Motivic equivalence}, we obtain the complete classification of $p'$-inner $p$-consistent groups up to \emph{motivic equivalence}, by means of their higher Tits $p$-indexes. These results expound how higher isotropy of such reductive groups determines motives of the associated projective homogeneous varieties.

\section{Notation}
\label{nota.section}

Let $F$ be a field and $\bar F$ a separable closure of $F$.
We denote by $\Gamma_F$ the absolute Galois group $\Gal(\bar F/F)$ of $F$.
Throughout the paper, $p$ is a prime number, $\F:=\Z/p\Z$, and $\Ch(\cdot)$ denotes the Chow group with coefficients in $\F$.
We also let $\mathrm{CM}(F,\F)$ be the category of Chow motives over $F$ with coefficients in $\F$
(see \cite[\S64]{EKM}), while $\CMe(F,\F)$ stands for its full subcategory of effective motives.
A {\em variety} is a separated scheme of finite type over a field.
For any smooth projective $F$-variety $X$, we denote by $M(X)$ its motive in both categories.
We use the notation $\AM(F,\F)$ for the full subcategory of $\CMe(F,\F)$ which consists of direct summands of motives of $0$-dimensional varieties, that is the category of Artin (Chow) motives, see \S\ref{Artin.section}.

A {\em complete decomposition} of a motive $M$ is a finite direct sum decomposition with indecomposable summands.
We say that the Krull-Schmidt property holds for a motive $M$ if any direct sum decomposition of $M$ can be refined into a complete one, and $M$ admits a unique complete decomposition up to permutation and isomorphism of the summands. Since we work with finite coefficients, this property holds for direct summands of motives of geometrically split varieties satisfying the nilpotence principle, see~\cite[\S 2.I]{upper}. This covers, in particular, the case of projective homogeneous varieties under the action of a reductive algebraic group.

Given an $F$-variety $X$ and a field extension $L/F$, $X_L$ is the $L$-variety given by the product of the $F$-schemes $X$ and $\Spec L$;
we also let $\bar X=X_{\bar F}$.
The functor $X\mapsto X_L$ for smooth projective $X$ extends to motives;
given an $F$-motive $M$, we write $M_L$ for the corresponding $L$-motive.

If $L/F$ is finite and $Y$ is an $L$-variety,
we let $Y^F$ be the $F$-variety given by the scheme $Y$ endowed with the composition $Y\to\Spec L\to\Spec F$.
In practice, we will only consider smooth projective varieties $Y$ and finite separable field extensions $L/F$, in which case the $F$-variety $Y^F$ is also smooth and projective.
Under these assumptions, by \cite[Theorem 2.1]{outer}, the Krull-Schmidt property holds for the motive of $Y^F$.
In particular, taking $Y=\Spec L$, one sees that the Artin motives have the Krull-Schmidt property.
By \cite[\S3]{outer},
the functor $Y\mapsto Y^F$ extends to motives,
the resulting functor $\CM(L,\F)\to\CM(F,\F)$ is called the {\em corestriction} functor;
given an $L$-motive $M$, we write $M^F$ for the corresponding $F$-motive.

By default, the spectrum of a field is the variety over this very field; for a finite field extension $L/F$, we use the notation $(\Spec L)^F$ for the $F$-variety given by the spectrum of $L$.
We write $M(L)^F$ for the motive of $(\Spec L)^F$, and $\F=M(F)=M(\Spec F)$ for the Tate motive. A motive $M$ is called geometrically split if the $\bar F$-motive $\bar M$ is isomorphic to a sum of shifts of $\F=M(\bar F)$; the number of summands is called the rank of $M$.

Let $E/F$ be a Galois field extension with Galois group denoted by $\Gamma$.
Given an $E$-variety $Y$ and an automorphism $\gamma\in \Gamma$, we write $Y_{\gamma}$ for the $E$-variety obtained from $Y$ by the base change via $\gamma$.
Thus $Y_\gamma$ is the scheme $Y$ viewed as an $E$-variety via the composition
\[
Y\to\Spec E\stackrel{\gamma^{-1}}{\longrightarrow}\Spec E,
\]
for which we also write $Y^{\gamma^{-1}}=Y_\gamma$.
This base change is invertible: $(Y_\gamma)_{\gamma^{-1}}=Y$. Hence the variety $Y_\gamma$ has a rational point if and only if $Y$ has one.
We use the same notation for motives, and for the same reason, a motive $M_\gamma$ is indecomposable if and only if $M$ is.

Let $X$ and $X'$ be smooth connected projective $F$-varieties. As in~\cite[\S 2]{motequiv}, we say that $X$ dominates $X'$ if there exists a correspondence $X\corr X'$ of multiplicity $1$. We call $X$ and $X'$ equivalent if they dominate each other, see~\cite[\S 2]{TateTraces}. Note that both relations depend on the choice of the prime number $p$; we may sometimes emphasize this by adding the expression (mod $p$). We will use the following observation on the behavior of the equivalence relation:
\begin{rem}
\label{cores.rem}
Let $L/F$ be a finite separable field extension and consider smooth connected projective varieties $X/F$ and $Y/L$.
If the degree of $L/F$ is prime to $p$, then $X_L$ and $Y$ are equivalent if and only if $X$ and $Y^F$ also are.
Indeed, as schemes, we have $X_L\times Y=X\times Y^F$, therefore a given element in $\Ch_{\dim X}(X_L\times Y)$ can be viewed as a correspondence $\alpha: X_L\corr Y$ as well as $\alpha':X\corr Y^F$. The multiplicity is given in each case by the push-forward of the first projection, so we have $\mult(\alpha')=[L:F]\mult(\alpha)$. Similarly, viewing
an element in $\Ch_{\dim Y}(Y\times X_L)$ as correspondences $Y\corr X_L$ and $Y^F\corr X$, we get that they have the same multiplicity.
\end{rem}
We will also use the following, where $R_{L/F}$ denotes the Weil restriction functor:
\begin{lemma}
\label{R}
Let $L/F$ be a finite separable field extension, and $Y/L$ a smooth connected projective variety. We assume there exists an $F$-variety $\hat Y$ such that $\hat Y_L$ and $Y$ are equivalent. If the degree of $L/F$ is prime to $p$, then
the $F$-varieties $Y^F$ and $R_{L/F}(Y)$ are equivalent.
\end{lemma}

\begin{proof}
By Remark~\ref{cores.rem}, it is enough to prove that the $L$-varieties $Y$ and $R_{L/F}(Y)_L$ are equivalent.
Since the functor $R_{L/F}$ is right adjoint to the base change functor (see \cite[(4.2.2)]{MR1321819}),
we have
$$
\Mor(R_{L/F}(Y)_L,Y)=\Mor(R_{L/F}(Y),R_{L/F}(Y))\ni\id
$$
showing that there is a morphism $R_{L/F}(Y)_L\to Y$ and, in particular, $R_{L/F}(Y)_L$ dominates $Y$.

To obtain the other domination, we use the variety $\hat{Y}$.
Applying the Weil transfer of \cite[\S3]{MR1809664} to multiplicity one correspondences between $\hat{Y}_L$ and $Y$,
we get multiplicity one correspondences between the varieties $R_{L/F}(\hat{Y}_L)$ and $R_{L/F}(Y)$ witnessing their equivalence.
Since
$$
\Mor(\hat{Y},R_{L/F}(\hat{Y}_L))=\Mor(\hat{Y}_L,\hat{Y}_L)\ni\id,
$$
there is a morphism $\hat{Y}\to R_{L/F}(\hat{Y}_L)$.
It follows that $\hat{Y}$ dominates $R_{L/F}(Y)$ and therefore
$\hat Y_L$, which is equivalent to $Y$, dominates $R_{L/F}(Y)_L$.
\end{proof}

\section{Artin motives}
\label{Artin.section}

A-upper motives, defined in \S\ref{Aupper.section}, are an essential tool in this paper.
The letter ``A'' in their name indicates their relationship with the {\em Artin motives}.
In this section, we recall known facts about Artin motives, and we provide an explicit example showing that the classification of Chow motives using their higher Tate trace provided by~\cite[Thm 4.3]{TateTraces} is not valid anymore if we relax the assumptions, see Remark~\ref{counterexample}.

By definition (cf.\! \cite[Definitions 1.2, 1.3]{MR3642472}), an Artin motive over $F$ is a direct summand in the Chow motive of the spectrum of an {\em \'etale $F$-algebra}
(that is up to isomorphism, a finite direct product of finite separable field extensions of $F$).
An Artin-Tate motive is a Tate shift of an Artin motive.\footnote{Thanks to Stefan Gille
and Alexander Vishik
for suggestion to consider the Artin and Artin-Tate motives in this context.}
In particular, the Tate motive $\F$, as well as $M(L)^F$ for all finite separable field extensions $L/F$, are Artin motives.
Artin motives form an additive subcategory of $\CMe(F,\F)$, denoted by $\AM(F,\F)$.
Here is a simple example of an indecomposable Artin motive that is not isomorphic to $\F$:

\begin{example}[{see Example \ref{ex3.4} for more details}]
\label{start}
Consider an odd prime number $p$, a field $F$, and a separable quadratic field extension $L/F$.
In $\CMe(F,\F)$ (as well as in $\CM(F,\F)$), the complete decomposition of the motive $M(L)^F$
consists of two summands:
$\F$ and $A$, where the motive $A$ (whose isomorphism class is uniquely determined by the
Krull-Schmidt property) satisfies
$$
\Hom(\F,A)=0=\Hom(A,\F).
$$
In particular, $A$ is not isomorphic to $\F$.
\end{example}

Artin motives may be described in terms of Galois permutation modules, as we now proceed to recall.
Note that for any $F$-variety $X$, the Chow group $\Ch^0(\bar{X})$ is a continuous $\Gamma_F$-module.
To better understand the structure of this module when $X$ is an Artin motive,
one can use the anti-equivalence of categories between \'etale $F$-algebras
and finite sets with a continuous left $\Gamma_F$-action, see \cite[(18.4)]{MR1632779}.
The $\Gamma_F$-set corresponding to an \'etale $F$-algebra $L$ is the set of $F$-algebra homomorphisms from $L$ to $\bar F$.
Its cardinality is equal to the dimension of $L$ over $F$.
The direct product (respectively, tensor product) of \'etale $F$-algebras corresponds to the disjoint union (respectively, direct product) of $\Gamma_F$-sets.
Note that $L$ is a field if and only if the corresponding $\Gamma_F$-set is transitive.

Let $L\subset \bar F$ be a finite separable field extension of $F$ embedded into $\bar{F}$, and let $\Gamma_L=\Gal(\bar F/L)$.
The set of $F$-algebra homomorphisms from $L$ to $\bar F$ is identified with the set of left cosets $\Gamma_F/\Gamma_L$, on which $\Gamma_F$ acts by left multiplication.
For the $F$-variety $X:=(\Spec L)^F$, consider the $\bar{F}$-variety $\bar X=\Spec(L\otimes_F\bar F)$.
Using the identification of $\bar F$-algebras
\begin{equation}
\label{gal.eq}
L\otimes_F\bar F = \Prod_{\Gamma_F/\Gamma_L} \bar F,\ \ x\otimes \lambda\mapsto (\gamma(x)\lambda)_{\gamma\Gamma_L\in \Gamma_F/\Gamma_L},
\end{equation}
we see that $\bar X$  is a disjoint union of base points identified with $\Gamma_F/\Gamma_L$.
To the Artin motive $M(L)^F$, we associate the Chow group $\Ch^0(\bar X)$, which is a
transitive permutation $\F[\Gamma_F]$-module isomorphic to the $\F[\Gamma_F]$-module
$\F[\Gamma_F/\Gamma_L]$ given by the $\Gamma_F$-set $\Gamma_F/\Gamma_L$.
(The isomorphim depends on the choice of the embedding $L\hookrightarrow\bar{F}$.)
By {\em permutation module} over a group ring $\F[\Gamma]$ we mean a module possessing a finite base over $\F$ permuted by
$\Gamma$; in particular, all our permutation modules are finite dimensional vector spaces over $\F$.

By the same arguments as in \cite[\S7]{MR2264459}, we obtain an anti-equivalence of additive categories between the category
of Artin motives $\AM(F,\F)$ and the category of direct summands in continuous permutation $\F[\Gamma_F]$-modules.
This anti-equivalence is compatible with the tensor products in these two categories.
Indeed, the tensor product of two \'etale $F$-algebras corresponds to the direct product of the associated $\Gamma_F$-sets, which in turn gives rise to the tensor product of the corresponding permutation modules.

\begin{rem}
Restricting the duality functor $\CM(F,\F)\to\CM(F,\F)^{\mathrm{op}}$ of \cite[\S65]{EKM} to the subcategory $\AM(F,\F)\subset\CM(F,\F)$, we get a functor $\AM(F,\F)\to\AM(F,\F)^{\mathrm{op}}$ which is identity on the motives of varieties.
Composing it with the above anti-equivalence, one gets the equivalence
of additive categories between the category
of Artin motives $\AM(F,\F)$ and the category of direct summands of continuous permutation $\F[\Gamma_F]$-modules,
obtained in \cite[\S7]{MR2264459} directly using the Chow functor $\Ch_0$ in place of $\Ch^0$.
\end{rem}

By construction, the motive $M(L)^F$ corresponds to the permutation module $\F[\Gamma_F/\Gamma_L]$; in particular,
the Tate motive $\F=M(F)$ corresponds to $\F[\Gamma_F/\Gamma_F]$, that is
the $1$-dimensional module $\F$ with trivial $\Gamma_F$-action.

Let $E\subset \bar F$ be a finite Galois field extension of $F$ containing $L$,
and let $\Gamma$ be its Galois group $\Gal(E/F)$.
The action of $\Gamma_F$ on $\F[\Gamma_F/\Gamma_L]$ factors through $\Gamma$.
We may consider $\F[\Gamma_F/\Gamma_L]$ as an $\F[\Gamma]$-module rather than an $\F[\Gamma_F]$-module
(without affecting, say, its endomorphism ring).
This applies notably when $L/F$ itself is Galois and $E=L$.

\begin{example}
\label{ex3.4}
In the settings of Example~\ref{start}, we have $\Gamma_F/\Gamma_L=\Gamma=\{1,\sigma\}\simeq {\mathbb Z}/2{\mathbb Z}$.
The $\F[\Gamma]$-module $\F[\Gamma]$, associated to $M(L)^F$, decomposes as
$$
\F[\Gamma]= \F\cdot(1+\sigma)\oplus\F\cdot(1-\sigma).
$$
The action of $\Gamma$ is trivial on the first summand, and nontrivial on the second one.
So $\F\cdot(1+\sigma)$ corresponds to the Tate summand $\F$
in $M(L)^F$
whereas $\F\cdot(1-\sigma)$ corresponds to $A$.
\end{example}

\begin{example}
\label{Artin}
The previous example can be extended as follows.
Let $p$ be an arbitrary prime number.
Consider a finite Galois field extension $L/F$ of some degree $n$ prime to $p$.
The $\F[\Gamma]$-module $\F[\Gamma]$ contains a submodule of dimension $1$ over $\F$ with trivial $\Gamma$-action, namely, $\F\cdot(\Sum_{\gamma\in\Gamma}\gamma)$.
Since $n$ is invertible in $\F$, this submodule splits off as a direct summand, where the complementary summand is
given by the submodule $B$ of linear combinations
$\Sum_{\gamma\in\Gamma}\lambda_\gamma \cdot\gamma$
satisfying $\Sum_{\gamma\in\Gamma}\lambda_\gamma=0$.
As a result, $M(L)^F$ contains an indecomposable direct summand isomorphic to the Tate motive $M(F)=\F$. We get a direct sum decomposition $M(L)^F=\F\oplus A$, where the Artin motive $A$ corresponds to the $\F[\Gamma]$-module $B$.

(i)  Assume $p=2$ and $n=3$, so that $L/F$ is a cubic field extension. The $\F[\Gamma]$-module $B$ has no proper stable submodule in this case, so that $M(L)^F=\F\oplus A$ with $A$ indecomposable. Over $L$, the motive $A$ is isomorphic to $\F\oplus\F$.

(ii) Assume now $p=7$ and $n=3$. Pick a generator $\sigma$ of $\Gamma$. The module $B$ admits a basis given by the elements $v_1:=1+2\sigma-3\sigma^2$ and $v_2:=1-3\sigma+2\sigma^2$, which satisfy $\sigma v_1=-3v_1$ and $\sigma v_2=2v_2$. Therefore, $B=B_1\oplus B_2$ with $B_i:=F\cdot v_i$,
and $A=A_1\oplus A_2$ with $A_i$ corresponding to $B_i$.
The motives $A_1$ and $A_2$ are indecomposable Artin motives, non-isomorphic to $\F$ over $F$ and becoming isomorphic to $\F$ over $L$.
Moreover, since $\sigma$ acts on $\F v_1$ and $\F v_2$ by multiplication by two different scalars, the modules $B_1$ and $B_2$ are not isomorphic, so that the motives $A_1$ and $A_2$ are not isomorphic.
The action of $\Gamma$ on the tensor products
$B_1^{\otimes3}$, $B_2^{\otimes3}$, and $B_1\otimes B_2$ is trivial.
Therefore each of the motives $A_1^{\otimes3}$, $A_2^{\otimes3}$, and $A_1\otimes A_2$
is isomorphic to $\F$, i.e., the motives
$A_1$ and $A_2$ are invertible, and their classes in the Picard group of isomorphism classes of invertible motives in
$\CM(F,\F)$
(with multiplication induced by tensor product, \cite[Definition A.2.7]{MR1388895}) are inverse to each other elements of order $3$.
\end{example}

\begin{rem}
\label{counterexample}
The motives $A_1$ and $A_2$ defined in Example~\ref{Artin}(ii) have the same {\em Tate trace}
(see~\cite[Definition 3.5]{TateTraces})
over any field extension of $F$, even though they are not isomorphic over $F$.
Indeed, none of the indecomposable motives $A_1$ and $A_2$ is isomorphic to $\F$, hence each of them has trivial Tate trace.
This remains true over any field extension $K$ of $F$ such that the tensor product $L\otimes_F K$ is a field.
On the contrary, if $L\otimes_F K$ is not a field, it is the split \'etale $K$-algebra $K\times K\times K$.
Hence the motive $M(L)^F$ becomes isomorphic to a direct sum of three copies of $\F$ over such $K$, so that both $A_1$ and $A_2$ become isomorphic to $\F$
over $K$.
This example demonstrates limitations for possible generalizations of Theorem 4.3 of~\cite{TateTraces},
and is a strong motivation for introducing the {\em Artin-Tate traces}, see Definition~\ref{ATTrace.dfn}.
\end{rem}

We finish this section with two results which aim at describing how Artin motives behave under base field extension to the function field of a geometrically integral variety. They are extensively used in the remaining part of the paper.

\begin{lemma}
\label{FX}
Let $X$ be a geometrically integral variety over a field $F$ and let $E/F$ be a finite Galois field extension.
Then $E(X)/F(X)$ is also a finite Galois field extension and its Galois group $\tilde{\Gamma}$ is isomorphic to
$\Gamma:=\Gal(E/F)$.
\end{lemma}

\begin{proof}
The extension $E(X)/F(X)$ is algebraic, normal, and separable;
therefore it is Galois.
Since $E$ is algebraically closed in $E(X)$, any element of
$\tilde{\Gamma}$
maps $E$ to $E$.
Together, the subfields $E$ and $F(X)$ generate the field $E(X)$; it follows that
the group homomorphism
$\tilde{\Gamma}\to\Gamma$, $\sigma\mapsto \sigma|_E$
is injective. It also is surjective; indeed, any element of $E$ which is fixed under the image of $\tilde{\Gamma}$ belongs to $E\cap F(X)=F$.
\end{proof}
Combining this Lemma with the anti-equivalence of categories between Artin motives and
direct summands of permutation modules, we get the following:
\begin{cor}
\label{AFX}

Let $X$ be a geometrically integral $F$-variety.
Let $L/F$ be a subextension of a finite Galois field extension $E/F$.
For any direct summand $\tilde{A}$ of the motive $M(L(X))^{F(X)}$, there is a direct summand
$A$ of $M(L)^F$ satisfying $A_{F(X)}\simeq\tilde{A}$.
The motive $\tilde{A}$ is indecomposable if and only if $A$ is.
Direct summands $A$ and $A'$ of $M(L)^F$ with isomorphic $A_{F(X)}$ and $A'_{F(X)}$ are isomorphic.
\qed
\end{cor}

\section{A retraction}
\label{Fm.section}
We now construct a functor $\Fm$ from the category $\CMe(F,\F)$ of effective Chow motives to its subcategory $\AM(F,\F)$ of Artin motives, which is a crucial ingredient in the definition of A-upper motives.

The category $\CMe(F,\F)$ is the idempotent completion of the category $\CC(F,\F)$ of degree $0$
Chow correspondences.
We first define a functor on $\CC(F,\F)$.

By definition, the objects of $\CC(F,\F)$ are given by smooth projective varieties over $F$;
we write $M(X)$ for the object given by such a variety $X$.
The morphisms from $M(X)$ to $M(X')$ are the degree $0$ correspondences $X\corr X'$ from $X$ to $X'$
with coefficients in $\F$,
where the degree of a correspondence is defined as in \cite[\S63]{EKM}.
(Note a difference with the definition of degree used in \cite{MR0258836}.)

Any smooth connected $F$-variety $X$ determines a finite separable field extension $L/F$ and a smooth geometrically connected $L$-variety $Y$ with $Y^F=X$.
The underlying scheme of the variety $Y$ is just the scheme of $X$.
The field $L$ coincides with the algebraic closure of $F$ inside the function field $F(X)$ of $X$ and is called the {\em field of constants} of $X$.
Let us choose an embedding $L\hookrightarrow\bar{F}$.
Since
\[\bar{X}=Y\times_{\Spec L}\Spec{(L\otimes_F \bar F)},\] the isomorphism~\eqref{gal.eq} provides an identification
\begin{equation}
\label{Xbar}
\bar X=\coprod_{\gamma\Gamma_L\in\Gamma_F/\Gamma_L} \bar{Y}_{\gamma},
\end{equation}
where $\bar Y_\gamma$ is $\bar{Y}$ modified by $\gamma\in\Gamma_F$, see the notation introduced in~\S\ref{nota.section}.
Note that since $Y$ is defined over $L$, we have $\bar Y_{\sigma}=\bar Y$ for all $\sigma\in\Gamma_L$ so that
$\bar{Y}_{\gamma}$ only depends on the coset $\gamma\Gamma_L$.
Therefore, the $\Gamma_F$-set of connected components of $\bar X$ is
identified with the $\Gamma_F$-set $\Gamma_F/\Gamma_L$ (the identification depends on the choice of the embedding
$L\hookrightarrow\bar{F}$).
It follows that $\Ch^0(\bar X)$ is a transitive permutation $\F[\Gamma_F]$-module isomorphic to $\F[\Gamma_F/\Gamma_L]$.

Dropping the assumption that $X$ is connected, we see that
$\Ch^0(\bar X)$ is the permutation $\F[\Gamma_F]$-module
$\Ch^0(\bar{X}_1)\oplus\dots\oplus\Ch^0(\bar{X}_n)$, where $X_1,\dots,X_n$ are the connected components of $X$.

\begin{lemma}
The additive contravariant functor from the category of correspondences $\CC(F,\F)$
to the category of abelian groups, which maps to $\Ch^0(\bar X)$
the motive $M(X)$ of a smooth projective $F$-variety $X$, yields an additive contravariant functor
from $\CC(F,\F)$ to the category of continuous permutation $\F[\Gamma_F]$-modules.
\end{lemma}

\begin{proof}
We already noticed that $\Ch^0(\bar X)$ is a continuous permutation $\F[\Gamma_F]$-module.
Consider a degree $0$ correspondence $\alpha:X\corr Y$ for smooth projective $F$-varieties $X$ and $Y$.
The induced homomorphism of abelian groups $\Ch^0(\bar Y)\mapsto \Ch^0(\bar X)$
is the composition
\begin{equation*}
\Ch^0(\bar Y)\stackrel{\bar p_2^*}{\longrightarrow}\Ch^0(\bar X\times \bar Y)\stackrel{\cdot \bar \alpha}{\longrightarrow}\Ch^d(\bar X\times \bar Y)\stackrel{\bar p_{1*}}{\longrightarrow}\Ch^0(\bar X),
\end{equation*}
where $d$ is the dimension of $Y$ and $p_1$ and $p_2$ are the projections from $X\times Y$ to $X$ and $Y$.
Since $p_1$, $p_2$, and $\alpha$ are defined over $F$, this composition commutes with the action of $\Gamma_F$,
hence respects the structure of $\F[\Gamma_F]$-module.
\end{proof}

Taking the idempotent completion of both categories, and combining with the anti-equivalence of categories between direct summands of
continuous permutation modules and Artin motives, described in~\S\ref{Artin.section}, we get an additive functor
\begin{equation}
\label{functor em}
\Fm\colon\CMe(F,\F)\to\AM(F,\F).
\end{equation}

We now prove some useful properties of this functor.

\begin{lemma}
\label{mult.lem}
The functor $\Fm$ maps the motive $M(X)$ of a smooth projective connected $F$-variety $X$
to the Artin motive $M(L)^F$, where $L$ is the field of constants of $X$.

If the field extension $L/F$ is Galois with Galois group $\Gamma$,
and $Y$ is the $L$-variety with $Y^F=X$,
the additive group of the ring $\End_{\CM(F,\F)}(M(X))$ is identified with
the direct sum
$\Oplus_{\sigma\in\Gamma}\Ch_{\dim Y}(Y_\sigma\times Y)$,
and $\Fm$ sends an element $\alpha\in\Ch_{\dim Y}(Y_\sigma\times Y)$ to
$$
\mult(\alpha)\cdot\sigma\in\F[\Gamma]=\End_{\AM(F,\F)}\bigl(M(L)^F\bigr),
$$
where $\mult(\alpha)$ is the {\em multiplicity} (see \cite[\S75]{EKM})
of the degree $0$ correspondence $\alpha\colon Y_\sigma\corr Y$.
\end{lemma}

\begin{proof}
Since $L/F$ is finite separable, we may assume $L\subset \bar F$; let $\Gamma_L=\Gal(\bar{F}/L)$.
The first assertion of Lemma \ref{mult.lem} is a direct consequence of the definition of $\Fm$, since, as noticed at the beginning of this section, $\Ch^0(\bar X)$ is isomorphic to the $\F[\Gamma_F]$-module $\F[\Gamma_F/\Gamma_L]$, which corresponds to the Artin motive $M(L)^F$.

Assume now $L/F$ is Galois. Using the identification
\begin{equation}
\label{LtensL}
L\otimes_FL = \Prod_{\Gamma} L,\ \ x\otimes y\mapsto (\sigma(x)y)_{\sigma\in\Gamma},
\end{equation}
we get that $X\times X=Y\times\Spec(L\otimes_F L)\times Y=\coprod_{\sigma\in\Gamma}Y_\sigma\times Y$.
Therefore, we have
$$
\End M(X)=\Ch_d(X\times X)=\Oplus_{\sigma\in\Gamma}\Ch_d(Y_\sigma\times Y),
$$
where $d$ is the dimension of $X$.

By \eqref{Xbar}, we have $\bar X=\coprod_{\sigma\in\Gamma}\bar Y_\sigma,$ hence
$
\Ch_d(\bar X\times \bar X)=\Oplus_{(\sigma,\sigma')\in\Gamma^2}\Ch_d\bigl(\bar Y_\sigma\times\bar Y_{\sigma'}\bigr).
$
To determine the image of $\alpha$ in $\Ch_d(\bar X\times\bar X)$, we use the following identifications:
$$
\bar X\times\bar X=(X\times X)\times_{\Spec F}\Spec{\bar F}=\coprod_{\sigma\in\Gamma}(Y_\sigma\times Y)\times\Spec(L\otimes \bar F).
$$
Using again the identification \eqref{gal.eq}, we get
$
\bar X\times \bar X=\coprod_{\sigma\in\Gamma}\bigl(\coprod_{\tau\in\gamma} \bar Y_{\sigma\tau}\times \bar Y_\tau\bigr).
$
Hence, an element $\alpha\in\Ch_d(Y_\sigma\times Y)\subset \Ch_d(X\times X)$ satisfies
$
\bar \alpha\in\Oplus_{\tau\in\Gamma}\Ch_d(\bar Y_{\sigma\tau}\times \bar Y_\tau)\subset \Ch_d(\bar X\times \bar X).
$
The endomorphism of $M(L)^F$, induced by $\alpha$, corresponds to the endomorphism of $\F[\Gamma]$-modules defined by
$$
\Ch^0(\bar X)\stackrel{\bar p_2^*}{\longrightarrow}\Ch^0(\bar X\times \bar X)\stackrel{\cdot \bar \alpha}{\longrightarrow}\Ch_d(\bar X\times \bar X)\stackrel{\bar p_{1*}}{\longrightarrow}\Ch^0(\bar X).
$$
The intersection product of the image under $p_2^*$ of  $[\bar Y]\in\Ch_0(\bar X)$ with $\bar \alpha$ is the projection of $\bar\alpha$ to the summand $\Ch_d(\bar Y_\sigma\times\bar Y)$.
This element maps under $p_{1*}$ to $\mult(\alpha)[\bar Y_\sigma]$.
Identifying $\Ch^0(\bar X)$ with $\F[\Gamma]$, we get that the endomorphism of $\Ch^0(\bar X)$, induced by $\alpha$,
maps $1$ to $\mult(\alpha)\cdot \sigma$; hence, it is the left multiplication by $\mult(\alpha)\cdot\sigma$, as claimed.
\end{proof}

\begin{rem}
\label{mult.rem}
If $X$ is geometrically connected, its field of constants is $F$, and we get that $\Fm(M(X))=\F$.
The homomorphism $\End M(X)\to\End\F$
of the endomorphism rings is the multiplicity homomorphism
$\Ch_{\dim X}(X\times_F X)\rightarrow \F$.
\end{rem}

Recall that for a finite separable field extension $K/F$ and a $K$-motive $M$,
we denote by $M^F$ its corestriction to $F$, defined as in~\cite[\S3]{outer}.

\begin{lemma}
\label{cor.lem}
The functor $\Fm$ commutes with the corestriction functor.
In particular, for every finite separable field extension $K/F$, and every motive $M\in\CM(K,\F)$, we have $\Fm(M^F)=\Fm(M)^F$.
\end{lemma}

\begin{proof}
Let $Y$ be a connected smooth projective $K$-variety, and let $L$ be its field of constants.
There exists an $L$-variety $Z$ such that $Y=Z^K$.
It follows that $Y^F=Z^F$.
Therefore,
\begin{equation*}
\Fm\bigl(M(Y)^F\bigr)=M(L)^F=\bigl(M(L)^K\bigr)^F=\Fm\bigl(M(Y)\bigr)^F
\end{equation*}
showing that the functor $\Fm$ commutes with the corestriction functor on objects.

To verify commutativity on morphisms, we consider a degree $0$ correspondence
$$
\alpha\colon Z_1^K\corr Z_2^K,
$$
where $Z_i$ for $i=1,2$  is a geometrically connected smooth projective
variety over a finite separable field extension $L_i$ of $K$.
Viewing $\alpha$ as a morphism $M(Z_1)^K\to M(Z_2)^K$ in the category $\CM(K,\F)$,
its corestriction $\beta\colon M(Z_1)^F\to M(Z_2)^F$ is given by the push-forward of
$\alpha\in\Ch(Z_1^K\times Z_2^K)$ with respect to the closed embedding
$Z_1^K\times Z_2^K\hookrightarrow Z_1^F\times Z_2^F$ (see \cite[\S3]{outer}).
The image $\Fm(\alpha)$ of the morphism $\alpha$ under the functor $\Fm$ is given by the homomorphism
of Galois modules $\Ch^0((Z_2^K)_{\bar{F}})\to\Ch^0((Z_1^K)_{\bar{F}})$ induced by
$\alpha$.
Similarly, the image $\Fm(\beta)$ of the morphism $\beta$ is given by the homomorphism
$\Ch^0((Z_2^F)_{\bar{F}})\to\Ch^0((Z_1^F)_{\bar{F}})$ induced by $\beta$.
Identifying the $\Gamma_F$-module $\Ch^0((Z_i^F)_{\bar{F}})$ with the image
of the $\Gamma_K$-module $\Ch^0((Z_i^K)_{\bar{F}})$ under the induction functor
of \cite[\S2.1]{MR3821516},
one identifies the homomorphism $\Fm(\beta)$
with the image of the homomorphism $\Fm(\alpha)$.
We finish by the observation that the corestriction functor $\AM(K,\F)\to\AM(F,\F)$ is also given by the induction functor on the categories of Galois modules.
\end{proof}

\begin{rem}
\label{Fm-res-tens.rem}
The functor $\Fm$ commutes as well with the restriction functors (given by arbitrary base field extensions) and respects
tensor products.
\end{rem}

\begin{rem}
\label{retraction}
The restriction of $\Fm$ to the subcategory of Artin motives is isomorphic to the identity functor. So, we may view $\Fm$ as a retraction of the category of effective Chow motives to its subcategory of Artin motives.
\end{rem}

\section{A-upper motives}
\label{Aupper.section}
Throughout this section, $X$ denotes a quasi-homogeneous $F$-variety in the sense of~\cite[\S 2]{outer}; each connected component of $X$ is the corestriction $Y^F$ of some projective homogeneous variety $Y/L$ under the action of a reductive algebraic group, where $L$ is a separable finite field extension of $F$.
In particular, the Krull-Schmidt property holds for $M(X)$, see~\cite[Cor. 2.2]{outer}.

If $X$ is connected, then any complete decomposition of $M(X)$ contains a unique indecomposable summand $P$ such that $\Ch^0(P)\not=0$; it is also the unique indecomposable summand defined by an idempotent of multiplicity $1$, see~\cite[Lem. 2.8]{MR2110630}. This motive $P$ is denoted by $U(X)$ and called the \emph{upper motive} of $X$. More generally, a summand $M$ of $M(X)$ is called upper if $\Ch^0(M)\not =0$ and non-upper otherwise.

The A-upper motives of a quasi-homogeneous variety $X$ are defined as follows:
\begin{dfn}
\label{Aupper.def}
Let $X$ be a quasi-homogeneous $F$-variety.
A motive $P\in\CM(F;\F)$ is called an A-upper motive of $X$ if it is isomorphic to an indecomposable summand of $M(X)$ and satisfies $\Ch^0(\bar P)\not =0$.
\end{dfn}
By definition of the functor $\Fm$, see \S\,\ref{Fm.section}, the A-upper motives of $X$ are the indecomposable summands $P$ of $M(X)$ with $\Fm(P)\not =0$.

\begin{example}
\label{A-upper.ex}
(i) The upper motive of a connected component $X^0=Y^F$ of $X$ is an A-upper motive of $X$. Indeed, if $p$ is an idempotent defining $U(X^0)$, $p$ has multiplicity $1$, so $\Fm(p)=1$ and $\Fm\bigl(U(X^0)\bigr)=\F$.

(ii) If $X$ is geometrically connected, $\Ch^0(\bar X)=\F$. So $X$ has a unique A-upper motive which is its upper motive.

(iii) Artin motives provide examples of A-upper motives that are not upper motives. By Remark~\ref{retraction}, given a finite separable field extension $L/F$, all indecomposable summands of $M(L)^F$ actually are A-upper summands. To give an explicit example, we use the notations of Example~\ref{Artin}(ii), where $M(L)^F=\F\oplus A_1\oplus A_2$. The motives $A_1$ and $A_2$ are A-upper motives of $M(L)^F$. This may also be checked directly using Remark~\ref{counterexample}, which shows that $\bar A_1=\bar A_2=\F$, so that $\Ch^0(\bar A_1)=\Ch^0(\bar A_2)=\F$ is nontrivial. Moreover, since $\Ch^0(M(L)^F)=\Ch^0(\F)=\F$, we have $\Ch^0(A_1)=\Ch^0(A_2)=0$.
\end{example}

In general, the A-upper motives of each connected component of a quasi-homogeneous variety can be described using the following proposition:
\begin{prop}
Let $X=Y^F$ be the corestriction of a projective homogeneous variety $Y$ under a reductive algebraic group, defined over a finite separable field extension $L/F$. We have:
\begin{enumerate}
\item $\Fm\bigl(M(X)\bigr)=\Fm\bigl(U(Y)^F\bigr)=M(L)^F$;
\item the A-upper motives of $X$ are the indecomposable summands of $U(Y)^F$ with nontrivial image under $\Fm$.
\end{enumerate}
\end{prop}

\begin{proof} Consider the upper motive $U(Y)$ in $\CM(L,\F)$ and its corestriction $U(Y)^F$ in $\CM(F,\F)$. As explained in~\cite[\S 3]{outer}, $U(Y)^F$ contains the upper motive $U(Y^F)$ of $Y^F$ as a direct summand. But in general, $U(Y)^F$ is not indecomposable, hence not isomorphic to $U(Y^F)$.
By Example~\ref{A-upper.ex}(i), we have $\Fm\bigl(U(Y)\bigr)=M(L)$.
Since $\Fm$ commutes with the corestriction functor by Lemma~\ref{cor.lem}, we get $\Fm\bigl(U(Y)^F\bigr)=M(L)^F$.
Consider now a non-upper indecomposable summand $Q$ of $M(Y)$. It is defined by an idempotent $q$ with multiplicity $0$, so it has trivial image under $\Fm$. It follows that any indecomposable summand of $Q^F$ for such a $Q$ also has trivial image under $\Fm$, and the proposition follows.
\end{proof} 
\begin{rem}
Writing $p$ for an idempotent defining the upper summand $U(Y)$ of $M(Y)$, we get the following commutative diagram, where $i$ is the natural inclusion, $j$ is the surjective map defined by $j(f)=p^Ffp^F$ for all $f\in\End\big(M(Y)^F\big)$, and the commutativity follows from the fact that $p^F$ maps to $1$ in $\End\bigl(M(L)^F\bigr)$.
\begin{equation}
\label{comdiag.eq}
\xymatrix
@=8pt{
\End\big(U(Y)^F\big)\ar@<1ex>[dd]^{i}\ar[drr]&&\\
&&\End\big(M(L)^F\big)\\
\End\big(M(Y)^F\big)\ar@<1ex>[uu]^{j}\ar[urr]&&\\
}
\end{equation}
The arrows $i$ and $j$ in the diagram are additive homomorphisms; moreover, $i$ respects multiplication, and maps $p^F$ which is the unit element of $\End\big(U(Y)^F\big)$ to the idempotent $p^F$ which is generally different from $1$ in $\End\big(M(Y)^F\big)$. The other two arrows are ring homomorphisms.
\end{rem}

Recall that two $F$-varieties $X_1$ and $X_2$ are equivalent (mod $p$) if there exist correspondences $X_1\corr X_2$ and $X_2\corr X_1$ both with multiplicity $1\in\F$, see~\cite[\S2]{TateTraces}. We now prove that, under some conditions on $Y$, the A-upper motives of $X$ are classified by some Artin motives, namely the indecomposable summands of $M(L)^F$.
\begin{thm}
\label{Aupper.thm}
Let $X=Y^F$ be the corestriction of a projective homogeneous variety $Y$ under a reductive algebraic group, defined over a finite separable field extension $L/F$. We assume that either $L/F$ is Galois, or the degree of its normal closure is prime to $p$.

\noindent (a) All nontrivial indecomposable summands of $U(Y)^F$ have nontrivial image under $\Fm$, and hence are A-upper motives of $X$.

\noindent (b) Assume in addition that there exists a projective homogeneous variety $\hat Y$ under a reductive algebraic group, defined over $F$, and such that $\hat Y_L$ is equivalent to $Y$ mod $p$. Then the following holds:
\begin{enumerate}
\item
\label{1}
every summand in $M(L)^F$ is isomorphic to the image under $\Fm$ of a summand in $U(Y)^F$;
\item
\label{2}
two summands in $U(Y)^F$ with isomorphic images under $\Fm$ are isomorphic;
\item
\label{3}
a summand in $U(Y)^F$ is indecomposable if and only if its image under $\Fm$ is so.
\end{enumerate}
\end{thm}
\begin{notation}
\label{A-upper.not}
Under the hypothesis of Theorem~\ref{Aupper.thm}\,(b), any indecomposable summand $A$ of $M(L)^F$ corresponds to an A-upper motive of $X=Y^F$. It is uniquely defined up to isomorphism, and will be denoted by $U_A(Y)$.
Note that the base field $F$ of the motive $U_A(Y)$ does not show up in this notation
because it is concealed in the motive $A$.
\end{notation}

With this notation in hand, Theorem \ref{Aupper.thm}\,(b) implies that, under its hypothesis,
given a complete decomposition $M(L)^F=A_1\oplus\dots\oplus A_r$, we get a complete decomposition
$U(Y)^F=U_{A_1}(Y)\oplus\dots\oplus U_{A_r}(Y)$, and the summands $U_{A_i}(Y)$ are the A-upper motives of $X=Y^F$.

The rest of the section is devoted to the proof of Theorem \ref{Aupper.thm}.

\begin{proof}[Proof of Theorem \ref{Aupper.thm} in the Galois case]
We assume that $L/F$ is Galois and we write $\Gamma$ for its Galois group $\Gamma=\Gamma_F/\Gamma_L$.

Let us first prove (a).
By \eqref{LtensL} we have
\[(Y^F)_L=\coprod_{\sigma\in\Gamma}Y_\sigma.\]
We claim that
\begin{equation}
    \label{rescor.eq}
\bigr(U(Y)^F\bigl)_L\simeq\Oplus_{\sigma\in\Gamma}U(Y)_\sigma=\Oplus_{\sigma\in\Gamma}U(Y_\sigma).
\end{equation}
Indeed, consider a complete decomposition $M(Y)=U(Y)\oplus Q_1\dots\oplus Q_r$ of $M(Y)$. By definition of the upper motive, we have $\Ch^0(Q_i)=0$ for $1\leq i\leq r$, hence $\Ch^0(Q_i^F)=0=\Ch^0\bigl((Q_i^F)_L\bigr)$, see~\cite[\S 3]{outer}. It follows that $\Oplus_{\sigma\in\Gamma}U(Y_\sigma)$ is a summand of
$\bigr(U(Y)^F\bigl)_L$. Over $\bar F$, both are isomorphic to $[L:F]$ copies of $\overline{U(Y)}$, and by the nilpotence principle, this proves~\eqref{rescor.eq}.

Let us now consider a nontrivial projector $p\in \End\big(U(Y)^F\big)$, defining a direct summand $M$ of $U(Y)^F$. We need to prove that its image $\Fm(p)=q\in\End\big(M(L)^F\big)$ is nontrivial.
The Krull-Schmidt property shows that $M_L=\Oplus_{\sigma\in S}U(Y_\sigma)$ for some nonempty subset $S$ of $\Gamma$.
By Lemma~\ref{mult.lem}, $\Fm(U(Y_\sigma))$ is the copy of $M(L)$ indexed by $\sigma$ in $$\Fm\big(M(Y^F)_L\big)=\Oplus_{\sigma\in\Gamma}M(L).$$ Therefore, $\Fm(M)$ is nontrivial, as required.\\

To prove (b), we use the following:
\begin{lemma}
\label{nilkernel}
The ring homomorphism \[\Fm:\,\End\big(U(Y)^F\big)\rightarrow\End\big(M(L)^F\big)\]
given by the functor $\Fm$ is surjective; its kernel consists of nilpotents.
\end{lemma}
\begin{proof}
The second assertion follows from (a): consider $f\in\End\bigl(U(Y)^F\bigr)$ and assume $\Fm(f)=0$. Since we work with finite coefficients, by \cite[Corollary 2.2]{upper}, some power of $f$ is a projector $q$, which satisfies $\Fm(q)=0$. Therefore (a) asserts $q=0$, and this shows $f$ is nilpotent.

Let us assume now that there exists a variety $\hat Y$ as in the statement of Theorem~\ref{Aupper.thm}\,(b).
Since $Y$, hence also $Y_\sigma$, is equivalent to $\hat{Y}_L$
for any $\sigma\in \Gamma$, there exists a multiplicity $1$ correspondence $Y_\sigma\corr Y$.
As explained in Lemma~\ref{mult.lem}, we may view it as an element of $\End\big(M(Y)^F\big)$,
and it maps under $\Fm$ to $\sigma\in\F[\Gamma]$.
This proves that $\End\big(M(Y)^F\big)$ maps surjectively onto $\End\big(M(L)^F\big)=\F[\Gamma]$. Lemma~\ref{nilkernel} follows by diagram~\eqref{comdiag.eq}.
\end{proof}
\medskip

With this in hand, assertions (1), (2) and (3), in the Galois case, are proved as follows.

\begin{enumerate}
    \item By Lemma \ref{nilkernel}, the projector $p$ defining a given summand in $M(L)^F$ lifts to an element of
$\End \big(U(Y)^F\big)$.
By \cite[Corollary 2.2]{upper}, an appropriate power of this element is a projector which maps to $p$ under $\Fm$ .\\
\item Let $M_1$ and $M_2$ be summands of $U(Y)^F$.
Any morphism between $\Fm(M_1)$ and $\Fm(M_2)$ is given by an endomorphism of $M(L)^F$ and therefore,
by Lemma \ref{nilkernel}, can be lifted to a
morphism between $M_1$ and $M_2$.
In particular, if $\Fm(M_1)$ and $\Fm(M_2)$ are isomorphic, mutually inverse isomorphisms lift to some morphisms
$f\colon M_1\to M_2$ and $g\colon M_2\to M_1$.
By Lemma \ref{nilkernel} once again, each of the compositions $g\compose f$ and
$f\compose g$
has the form $\id+\varepsilon$ with some nilpotent $\varepsilon$ and so is an isomorphism (with inverse given by
the finite sum $\id-\varepsilon+\varepsilon^2-\dots$). It follows that $f$ and $g$ are isomorphisms, even though they need not be mutually inverse. \\
\item This is a consequence of (1) and (2), since the functor $\Fm$ preserves direct sum decompositions.
\qedhere
\end{enumerate}
\end{proof}

\begin{proof}[Proof of Theorem \ref{Aupper.thm} in the non-Galois case]
We first prove the following result, which is of independent interest, and will be used in the proof:
\begin{lemma}
\label{indec.lem}
Let $Z/F$ be a projective homogeneous variety under the action of a reductive group.
Let $L/F$ be a finite separable field extension, and assume that $[E:F]$ is prime to $p$, where $E$ is a normal closure of $L/F$. Then then restriction $U(Z)_L$ of the upper motive of $Z$ is isomorphic to $U(Z_L)$. In particular, it is indecomposable.
\end{lemma}
\begin{proof}
Since $U(Z)_L$ is an upper summand in $M(Z_L)$, it contains $U(Z_L)$ as a direct summand; therefore, they are isomorphic if and only if $U(Z)_L$ is indecomposable. So, it is enough to prove that $U(Z)_E$ is indecomposable, or equivalently, isomorphic to $U(Z_E)$. Hence, we may assume that $L/F$ is Galois and $[L:F]$ is prime to $p$.
Consider the variety $Y=Z_L$.
The varieties $Y^F$ and $Z$ are equivalent (mod $p$) by Remark~\ref{cores.rem}, so Corollary 2.15 in~\cite{upper} shows that $U(Y^F)\simeq U(Z)$.
Hence, we may view $U(Z)$ as a direct summand of $U(Y)^F$, and $U(Z)_L$ as a direct summand of $\bigl(U(Y)^F)_L$.
Since $Y=Z_L$, we have $Y_\sigma\simeq Y$ for all $\sigma$ in the Galois group of $L/F$. It follows by ~\eqref{rescor.eq} that $U(Z)_L$ is isomorphic to a direct sum of copies of $U(Z_L)$.
Since $\Ch^0\bigl(U(Z)_L\bigr)=\F=\Ch^0\bigl(U(Z_L)\bigr)$, we get $U(Z)_L=U(Z_L)$ as expected.
 \end{proof}

We may now prove part (a) of Theorem~\ref{Aupper.thm}.
We continue to assume that the degree of normal closure of $L/F$ is prime to $p$.
Let $M$ be a nontrivial summand of $U(Y)^F$, and consider its restriction $M_E$, which also is nontrivial by the nilpotence principle. Since $E$ is the normal closure of $L/F$, $\gamma(L)\subset E$ for all $\gamma\in \Gamma_F$, and we have $L\otimes_F E=\prod_{\Gamma_F/\Gamma_L} E$, as in~\eqref{gal.eq}. Therefore, the same computations as in the Galois case show that
\[\bigl(U(Y)^F\bigr)_E\simeq \bigoplus_{\gamma\Gamma_L\in\Gamma_F/\Gamma_L}(U(Y)_E)_\gamma,\]
where $\bigl(U(Y)_E\bigr)_\gamma$ does not depend on the choice of $\gamma$ in its class, since $Y$ is defined over $L$. By Lemma~\ref{indec.lem}, $U(Y)_E=U(Y_E)$ is indecomposable, and so are the $\bigl(U(Y)_E\bigr)_\gamma=U\bigl((Y_E)_\gamma\bigl)$. It follows that $M_E$, which is a direct sum of some of these upper motives, has non-zero image under the functor $\Fm$, as required.

To prove part (b) of Theorem~\ref{Aupper.thm}, consider a projective homogeneous variety $\hat Y$ over $F$ such that $Y$ is equivalent to $(\hat Y)_L$ as in the statement. Applying again Remark~\ref{cores.rem} and~\cite[Cor. 2.15]{upper}, we get that $\hat Y$ is equivalent to $Y^F$, and $U(\hat Y)\simeq U(Y^F)$. It follows that $U(Y^F)_L\simeq U(\hat Y)_L$. Applying Lemma~\ref{indec.lem} to $\hat Y$, we get:
\begin{equation}
\label{upper.eq}
U(Y^F)_L\simeq U(\hat Y_L)\simeq U(Y).
\end{equation}
Using this, we now prove that the corestriction of the upper motive of $Y$ satisfies
\begin{equation}
U(Y)^F\simeq U(Y^F)\otimes M(L)^F.
\end{equation}
Indeed, since $U(Y)$ is an $L$-motive, we have $U(Y)\simeq U(Y)\otimes M(L)$. Taking the corestriction and using~\eqref{upper.eq}, we get
\[U(Y)^F\simeq \bigl(U(Y^F)_L\otimes M(L)\bigr)^F\simeq U(Y^F)\otimes M(L)^F,\]
where the second isomorphism is a particular case of the general formula $(M_L\otimes N)^F=M\otimes N^F$, which holds for any $F$-motive $M$ and $L$-motive $N$, with $L/F$ finite and separable.

Hence, any direct summand $A$ in $M(L)^F$ produces a direct summand $U(Y^F)\otimes A$ in $U(Y)^F$.
It satisfies $\Fm(U(Y^F)\otimes A)=\F\otimes A= A$, see Example~\ref{A-upper.ex} and Remark~\ref{Fm-res-tens.rem}.
To finish the proof, we will use the following:
\begin{lemma}
\label{indecArtin.lem}
We continue to assume that the degree of normal closure of $L/F$ is prime to $p$.
Under the hypothesis of Theorem~\ref{Aupper.thm}\,(b), the motive
$M=U(Y^F)\otimes A$ is indecomposable for any indecomposable summand $A$ of $M(L)^F$.
\end{lemma}
Assuming the lemma, we get that a complete decomposition $M(L)^F=A_1\oplus\dots\oplus A_k$ of the spectrum of $L$ yields a complete decomposition
\[U(Y)^F\simeq U(Y^F)\otimes M(L)^F\simeq \bigoplus_{i=A}^k \bigl(U(Y^F)\otimes A_i\bigr).\]
Assertions (1), (2) and (3) of the theorem follows immediately.
So, it only remains to prove Lemma~\ref{indecArtin.lem}.
For this, we consider motivic decompositions of $U(Y^F)\otimes A$ over various extensions of the base field $F$.

First, observe that the motive $A_E$, which is a direct summand of the motive of $L\otimes_F E$ is a direct sum of $\rk(A)$ copies of the Tate motive $\F=M(E)$, where $\rk(A)$ is the rank of $A$. Since in addition $U(Y^F)_E$ is indecomposable by Lemma~\ref{indec.lem}, we get that $\bigl(U(Y^F)\otimes A\bigr)_E$ is a direct sum of $r$ copies of the indecomposable motive $U(Y^F)_E$.

On the other hand, the variety $\hat Y$ is a projective homogeneous variety under the action of a reductive group $G$ defined over $F$. Let $X_G$ be the $F$-variety of Borel subgroups of $G$. Over $F(X_G)$, the group $G$ is quasi-split, so $U(Y^F)_{F(X_G)}=U(\hat Y)_{F(X_G)}$ is a direct sum of a Tate motive $\F=M\bigl(F(X_G)\bigr)$ and some positive shifts of some effective motives. Hence, $\bigl(U(Y^F)\otimes A\bigr)_{F(X_G)}$ decomposes as a direct sum of the motive $A_{F(X_G)}$ and a motive $B$ with $\Ch^0(B_K)=0$ for all field extensions $K/F(X_G)$.
By Corollary~\ref{AFX}, we know in addition that $A_{F(X_G)}$ is indecomposable.

Using these observations, let us now prove Lemma~\ref{indecArtin.lem}. Consider a nontrivial direct summand $U$ in $U(Y^F)\otimes A$. Over $E$, $U_E$ is a direct sum of $s$ copies of $U(Y^F)_E$, for some $s$ with $1\leq s\leq \rk(A)$, which also is the dimension over $\F$ of $\Ch^0(U_E)$.
Therefore, $\Ch^0(U_{E(X_G)})$ is non zero, and this implies that $A_{F(X_G)}$ is a direct summand of $U_{F(X_G)}$.
It follows that
\[s\geq \dim \Ch^0(A_{E(X_G)})=\rk(A).\]
So $U_E$ is a direct sum of $\rk(A)$ copies of $U(Y^F)_E$ and by the nilpotence principle, this implies $U=U(Y^F)\otimes A$. This proves that $U(Y^F)\otimes A$ is indecomposable.
\end{proof}
\begin{rem}
\label{Charles}
We continue to assume that the degree of normal closure of $L/F$ is prime to $p$.
    From the above proof, we get that under the assumptions of Theorem~\ref{Aupper.thm}\,(b), we have
    \begin{equation}
        \label{Charles.eq}
    U_A(Y)\simeq U(Y^F)\otimes A
    \end{equation}
    for any indecomposable summand $A$ in $M(L)^F$, where $U_A(Y)$ is as in Notation~\ref{A-upper.not}.
    Moreover, for any field extension $K/F$ such that $Y_K$ is irreducible and $A_K$ is indecomposable, we claim that the $K$-motive $\bigl(U_A(Y)\bigr)_K$ contains a summand isomorphic to $A_K$ if and only if $\bigl(U(Y^F)\bigr)_K$ contains a Tate motive $\F$. One implication follows immediately from~\eqref{Charles.eq}; to prove the converse, note that the indecomposable summands of $\bigl(U_A(Y)\bigr)_K$ are isomorphic to $P\otimes A_K$, where $P$ describes indecomposable summands of $U(Y^F)_K$. Comparing the ranks, if such a summand $P\otimes A_K$ is isomorphic to $A_K$, then $P$ has rank $1$.
    By~\cite[Lemma   2.21]{upper}, it follows that $(Y^F)_K$ has a $0$-cycle of degree $1\in\F$. So its upper motive, which is a summand of $U(Y^F)_K$, contains a summand isomorphic to $\F$ and this proves the claim.
\end{rem}

\section{Motivic decompositions}
\label{Decomposition}
Let $G$ be a reductive group defined over $F$. If $G$ is of inner type, indecomposable summands of a complete motivic decomposition of any projective $G$-homogeneous $F$-variety are described in~\cite[Theorem 3.5]{upper}. A generalization of this result dealing with groups of $p$-inner type, that is groups which become of inner type over a $p$-power field extension of $F$, is given in~\cite[Theorem 1.1]{outer}. In both cases, the description is in terms of upper motives of some projective homogeneous varieties.
The main result of this section is Theorem~\ref{main}, which provides another generalization of~\cite[Theorem 3.5]{upper} in a different direction. It uses the notion of A-upper motives of $G$, defined as follows:
\begin{dfn}
\label{AupperG}
Let $G$ be a reductive group over $F$ and let $E/F$ be a minimal field extension such that $G_E$ is of inner type. An \textit{A-upper motive of $G$} is an $F$-motive isomorphic to an A-upper motive of $Y^F$ for some projective $G_L$-homogeneous variety $Y$, defined over an intermediate field $L$ of $E/F$.
\end{dfn}
\begin{rem}
Note that for a given $G$, the field extension $E/F$ in Definition \ref{AupperG} is uniquely determined up to $F$-isomorphism so that its choice does not influence the notion of A-upper motives of $G$.
\end{rem}
The group $G$ is called $p'$-inner if the degree of $E/F$ is prime to $p$.

\begin{thm}
\label{main}
Let $G$ be a $p'$-inner reductive group defined over $F$.
Every summand in the complete decomposition of the Chow motive with coefficients in $\F=\Z/p\Z$ of any projective $G$-homogeneous variety $X$ is a Tate shift of an A-upper motive of $G$.
\end{thm}

\begin{proof}
Since the center of $G$ acts on $X$ trivially, we may assume that $G$ is semisimple and adjoint.

We write $\D_G$ (or simply $D$) for the set of vertices of the Dynkin diagram of $G$, and let $E/F$ be a minimal field extension with inner $G_E$.
The field extension $E/F$ is Galois and its
Galois group $\Gamma=\Gal(E/F)$ acts on $\D$.
For a field $L$ with $F\subset L\subset E$,
the set $\D_{G_L}$ is identified with $\D=\D_G$.
Any $\Gal(E/L)$-stable subset $\tau$ in $\D$ determines a projective $G_L$-homogeneous
variety $Y_{G_L,\tau}$ the way described in \cite[\S3]{upper} (which is opposite to the original convention of \cite[\S1.6]{MR0224710}).
For instance, $Y_{G_L,\D}$ is the variety of Borel subgroups of $G_L$, and $Y_{G_L,\emptyset}=\Spec L$.
Any projective $G_L$-homogeneous variety is isomorphic to $Y_{G_L,\tau}$ for some
$\Gal(E/L)$-stable $\tau\subset \D$.

We prove  Theorem \ref{main}
simultaneously for all $F,G,X$ using induction on $n:=\dim X$.
The base of the induction is $n=0$ where $X=\Spec F$ and the statement is trivial.

From now on we are assuming that $n\geq1$ and that Theorem
\ref{main} is already proven for varieties of dimension $<n$.

For any field extension $L/F$, we write $\tilde{L}$ for the function field $L(X)$
(note that any projective homogeneous variety and, in particular $X$, is geometrically integral).
Let $G'$ be the semisimple group over the field $\tilde{F}=F(X)$ given by the semisimple anisotropic kernel of the group $G_{\tilde{F}}$.
We note that the group $G'_{\tilde{E}}$ is of inner type.
By Lemma \ref{FX}, the field extension $\tilde{E}/\tilde{F}$ is Galois with Galois group
$$
\Gamma=\Gal(\tilde{E}/\tilde{F})=\Gal(E/F).
$$
In particular,
any of its intermediate fields is of the form $\tilde{L}$ for some intermediate field
$L$ of the extension $E/F$.
The set $\D_{G'}$ is identified with a $\Gamma$-invariant subset in $\D_G$;
the complement $\D_G\setminus\D_{G'}$ contains the subset in $\D_G$ corresponding to $X$.

Let $M$ be an indecomposable summand of the motive of $X$.
We are going to show that $M$ is isomorphic to a shift of a direct summand in $U(Y_{G_L,\tau})^F$
for some intermediate field $L$ of $E/F$ and some $\Gal(E/L)$-stable subset
$\tau\subset \D_G$ containing the complement of $\D_{G'}$.
This will prove Theorem \ref{main}.

According to \cite[Theorem 4.2]{MR2178658} (an enhancement of
\cite[Theorem 7.5]{MR2110630}),
the motive of $X_{\tilde{F}}$ decomposes into a sum of shifts of motives
of projective $G'_{\tilde{L}}$-homogeneous varieties $Y$,
satisfying $\dim Y<\dim X=n$, where $L$ runs over intermediate fields of
the extension $E/F$ by Lemma~\ref{AFX}.
It follows by the induction hypothesis
that each summand of the complete motivic decomposition of $X_{\tilde{F}}$ is a shift $N'\{i\}$ of
a summand $N'$ in $U(Y')^{\tilde{F}}$ for some $L/F\subset E/F$, some
$\Gal(E/L)$-stable $\tau'\subset\D_{G'}$,
and $Y':=Y_{G'_{\tilde{L},\tau'}}$.
By the Krull-Schmidt property \cite[Corollary 2.2]{outer},
the summands of the complete decomposition of $M_{\tilde{F}}$ are also of this shape.

In the complete decomposition of $M_{\tilde{F}}$,
let us choose a summand $N'\{i\}$ with minimal $i$.
It corresponds to a subset $\tau'\subset\D_{G'}$, and we set $\tau:=\tau'\cup(\D_G\setminus D_{G'})\subset\D_G$. The subset $\tau$ is $\Gal(E/L)$-stable.
To prove Theorem \ref{main},
it is enough to show that  $M$ is isomorphic to a direct summand in  $U(Y)^F\{i\}$
for these $L$, $\tau$, $i$, and $Y:=Y_{G_L,\tau}$.

Since $M$ is indecomposable,
it suffices to construct morphisms
\begin{equation}
\label{alpha-beta}
\alpha:U(Y)^F\{i\}\to M\;\;\text{and}\;\;
\beta:M\to U(Y)^F\{i\}
\end{equation}
such that no power of the composition $\alpha\compose\beta$ vanishes.
(We recall that by \cite[Corollary 2.2]{upper}, an appropriate power of any endomorphism of $M$
is a projector.)

We first construct certain, defined over the field $\tilde{F}$, predecessors
$\tilde{\alpha}$ and $\tilde{\beta}$ of $\alpha$ and $\beta$.
Recall that $N'\{i\}$ is a summand in $M_{\tilde{F}}$ and in $U(Y')^{\tilde{F}}\{i\}$.
Since $U(Y')=U(Y_{\tilde L})$ by~\cite[Cor. 2.15]{upper}, $U(Y')$ is a summand in $U(Y)_{\tilde L}$. Therefore, $U(Y')^{\tilde F}$ is a summand in
$$
(U(Y)_{\tilde L})^{\tilde F} = (U(Y)^F)_{\tilde F}.
$$
Using projections to and inclusions of direct summands, we define $\tilde{\alpha}$ and $\tilde{\beta}$ as the compositions
\begin{multline*}
\tilde{\alpha}\colon
U(Y)^F\{i\}_{\tilde{F}} \onto U(Y')^{\tilde{F}}\{i\}\onto N'\{i\}\hookrightarrow M_{\tilde{F}}
\;\;\text{ and }\;\;\\
\tilde{\beta}\colon M_{\tilde{F}}\onto N'\{i\}\hookrightarrow U(Y')^{\tilde{F}}\{i\}\hookrightarrow U(Y)^F\{i\}_{\tilde{F}},
\end{multline*}
where $\onto$ is a sign for a projection onto a direct summand
and $\hookrightarrow$ means an inclusion of a direct summand.
The composition $\tilde{\alpha}\compose\tilde{\beta}$ is the projector which yields the summand
$N'\{i\}$ of $M_{\tilde{F}}$.

Now we construct $\alpha$ and $\beta$ starting with $\alpha$.
Note that $\tilde{\alpha}$ is an element of the Chow group $\Ch(Y^F\times X)_{\tilde{F}}$ over $\tilde{F}$.
We take for $\alpha$ an element of the Chow group $\Ch(Y^F\times X)$ over $F$ such that
its image under the surjective ring homomorphism
$$
\Ch(Y^F\times X)\to\Ch(X_{F(Y^F)})
$$
(from \cite[Corollary 57.11]{EKM})
followed by the change of field homomorphism for the field extension $\tilde{F}(Y^F)/F(Y^F)$,
coincides with the image of $\tilde{\alpha}$ under the surjective ring homomorphism $$\Ch(Y^F\times X)_{\tilde{F}}\to\Ch(X_{\tilde{F}(Y^F)}).$$
Such $\alpha$ exists because the field extension $\tilde{F}(Y^F)/F(Y^F)$ is purely transcendental and therefore the
change of field homomorphism $\Ch(X_{F(Y^F)})\to\Ch(X_{\tilde{F}(Y^F)})$ is surjective as follows from the homotopy invariance of Chow groups
(see \cite[Theorem 57.13]{EKM} or \cite[Corollary 52.11]{EKM}) and \cite[Corollary 57.11]{EKM}).

In order to define $\beta$, we note that $\tilde{\beta}$ is an element of $\Ch(X\times Y^F)_{\tilde{F}}$ and let $\beta'$
be an element of $\Ch(X\times X\times Y^F)$ mapped to $\tilde{\beta}$ under the surjection (from \cite[Corollary 57.11]{EKM})
$$
\Ch(X\times X\times Y^F)\to\Ch(X\times Y^F)_{\tilde{F}}
$$
given by the generic point of the {\em second} factor in the product $X\times X\times Y^F$.
We consider $\beta'$ as a correspondence $X\corr X\times Y^F$ and let $\beta''$ be the composition of correspondences $\beta'\compose\mu$,
where $\mu\in\Ch(X\times X)$ is the projector which yields the motivic summand $M$ of $X$.
Finally, we define $\beta$ as the pullback of $\beta''$ with respect to the closed embedding
$$
X\times Y^F\hookrightarrow X\times X\times Y^F,\;\;(x,y)\mapsto(x,x,y)
$$
given by the diagonal of $X$.

Composing with the relevant idempotents, the elements $\alpha$ and $\beta$ we constructed induce morphisms as in~(\ref{alpha-beta}).
Changing notation, we write below $\alpha$ and $\beta$ for these two morphisms.
In particular, the composition $\alpha\compose\beta$ is an endomorphism of the motive $M$.
By \cite[Corollary 2.2]{upper},
an appropriate power $(\alpha\compose\beta)^{\compose r}$ of  this endomorphism is a projector
which defines a summand in $M$.
The morphisms $\beta$ and $\alpha\compose(\beta\compose\alpha)^{\compose (r-1)}$
identify this summand with a summand in $U(Y)^F\{i\}$ which we write in the form
$N\{i\}$ for certain summand $N$ in $U(Y)^F$.
By indecomposability of $M$, it suffices
to check that $N\ne0$ to conclude the proof.

Since $N'\ne0$, we have $\Fm(N')\ne0$ by Theorem \ref{Aupper.thm}(a).
In other terms, $\Ch^0(N'_{\check{F}})\ne0$, where $\check{F}$ is a separable closure of the field $\tilde{F}$.
So, the composition $\tilde{\alpha}\compose\tilde{\beta}$ yields a nonzero projector on $\Ch^0(Y^F)_{\check{F}}$.
By the construction of $\alpha$ and $\beta$, the action of the composition $\alpha\compose\beta$ on $\Ch^0(Y^F)_{\check{F}}$
coincides with the action of $\tilde{\alpha}\compose\tilde{\beta}$
(cf.\! \cite[Proof of Theorem 3.5]{upper}) and therefore also yields a nonzero projector.
Consequently, $\Fm(N)\ne0$ and it follows that $N$ is nonzero.
\end{proof}

\begin{rem}
Instead of \cite[Theorem 4.2]{MR2178658}, the weaker result
\cite[Theorem 7.5]{MR2110630} can be used in the proof of Theorem \ref{main}.
To do so, it suffices to take for $G'$ the semisimple part of the parabolic subgroup defining $X_{\tilde{F}}$.
\end{rem}

\begin{rem}
As follows from the proof of Theorem \ref{main},
the A-upper motives of $G$, whose Tate shifts actually appear as direct summands of $M(X)$ in Theorem \ref{main},
are associated with varieties $Y$ with $Y^F$ {\em dominating} $X$ in the sense of \cite[\S 2]{motequiv} (see also \cite[Lemma 2.2]{TateTraces}).
\end{rem}

\section{Classification results}

By~\cite[Thm. 4.3]{TateTraces}, using the Tate trace defined as a pure Tate summand of maximal rank of a motive, one may classify motives in a broad subcategory of the category of Chow motives $\CM(F,\F)$, with as usual $\F=\Z/p\Z$. This subcategory contains motives of projective homogeneous varieties under the action of a reductive group of $p$-inner type (i.e. of inner type over a $p$-power extension of $F$). On the other hand, this theorem does not apply when the reductive group is not $p$-inner, see Remark~\ref{counterexample} for an explicit counter-example using Artin motives.

The purpose of this section is to obtain a similar classification result in a different subcategory of $\CM(F,\F)$, which contains projective homogeneous varieties under the action of some reductive groups which are $p'$-inner, see~\S\ref{Decomposition} for a definition. A precise statement is given in Theorem~\ref{isomcrit} below.  To achieve this, we need to replace the Tate trace by the Artin-Tate trace of a motive, defined as the part of a complete decomposition of this motive which consists of Artin-Tate motives, i.e. shifts of Artin motives.

We first establish a criterion of isomorphism for A-upper motives of some reductive groups, a crucial tool in the proof of Theorem~\ref{isomcrit}. This can be done using Theorem~\ref{Aupper.thm}(b), by slightly restricting the class of reductive groups we are working with.
\begin{dfn}
\label{p-consistent}
    Let $G$ be a reductive group, and let $E/F$ be a minimal Galois field extension over which $G$ is of inner type.
The group $G$ is called {\em $p$-consistent} if for any intermediate field $L$ in $E/F$ and any projective homogeneous $G_L$-variety $Y$ over $L$, there exists a $G$-projective homogeneous variety $\hat Y$ over $F$ such that the $L$ varieties $\hat Y_L$ and $Y$ are equivalent (mod $p$).
\end{dfn}

In particular, if $G$ is $p'$-inner and $p$-consistent, the hypotheses of Theorem~\ref{Aupper.thm}(b) apply to varieties $Y$ as in the definition, so we may use the notation $U_A(Y)$ introduced in~\ref{A-upper.not} for its A-upper motives, where $A$ runs through indecomposable summands of the Artin motive $M(L)^F$.

Any non-$p$-inner absolutely simple group of type different from $^3\!\cat{D}_4$ and $^6\!\cat{D}_4$ is both $p'$-inner and $p$-consistent (for instance, the $3$-consistency of $\cat{E}_6$ follows from \cite{titspindexes}).
A direct product of $p'$-inner $p$-consistent groups is $p'$-inner and $p$-consistent.
Here is an additional source of $p'$-inner $p$-consistent groups:

\begin{example}\label{pexample}
Let $L/F$ be a {\em $p'$-extension}, i.e., a
finite separable field extension such that the degree of its normal closure is prime to $p$.
Given an inner reductive group $H$ over $F$, the group $G:=R_{L/F}(H_L)$, where $R_{L/F}$ is the Weil transfer,
is $p'$-inner and $p$-consistent.
\end{example}
The following theorem applies to A-upper motives (as defined in~\ref{Aupper.def} and~\ref{AupperG}) of all the groups listed above, and
extends~\cite[Corollary 2.15]{upper}.
\begin{thm}
\label{crit}
Let $U_A(Y)$ (respectively $U_{A'}(Y')$) be an A-upper motive of a $p'$-inner $p$-consistent reductive group $G$ (respectively $G'$) defined over $F$.
The motives $U_A(Y)$ and $U_{A'}(Y')$ are isomorphic in $\CM(F,\F)$ if and only if the Artin motives $A$ and $A'$ are isomorphic and the varieties $Y^F$ and $Y'^F$ are equivalent mod $p$.
\end{thm}
\begin{proof}
Let $E/F$ (respectively $E'/F$) be a minimal Galois extension of $F$ over which $G$ (respectively $G'$) is of inner type. By assumption, $Y$ is a projective $G_L$-homogeneous variety over $F$ for some intermediate field $L$ in $E/F$, $A$ is an indecomposable summand in $M(L)^F$ and $U_A(Y)$ is an indecomposable summand in $U(Y)^F$ with $\Fm\bigl(U_A(Y)\bigr)=A$. And similarly for $G'$, $Y'$, $A'$ and some intermediate field $L'$ in $E'/F$, with $\Fm\bigl(U_{A'}(Y')\bigr)=A'$. Applying the functor $\Fm$ to an isomorphism $U_A(Y)\simeq U_{A'}(Y')$ produces an isomorphism $A\simeq A'$. So we may assume $A$ and $A'$ are isomorphic.

By~\eqref{Charles.eq}, we have $U_A(Y)\simeq U(Y^F)\otimes A$ and $U_{A'}(Y')\simeq U(Y'^F)\otimes A'$.
Hence, if the varieties $Y^F$ and $Y'^F$ are equivalent, so that $U(Y^F)\simeq U(Y'^F)$ (see~\cite[Cor. 2.15]{upper}), we get $U_A(Y)\simeq U_{A'}(Y')$ as expected.

Assume conversely that $U_A(Y)\simeq U_{A'}(Y')$.
Since $G$ is $p$-consistent, there exists a $G$-projective homogeneous variety $\hat Y$ over $F$ such that $Y^F$ and $\hat Y$ are equivalent.
The Artin motive $A_{F(\hat Y)}\simeq A'_{F(\hat Y)}$ is a direct summand in $U_A(Y)_{F(\hat Y)}$, hence also in $U_{A'}(Y')_{F(\hat Y)}$.
By Remark~\ref{Charles},
this implies that the Tate motive $\F$ is a direct summand in $U(Y'^F)_{F(\hat Y)}$.
So, the variety $(Y'^F)_{F(\hat Y)}$ is isotropic (i.e., has a $0$-cycle of degree $1\in\F$), which means that
$Y^F$ dominates $\hat Y$ and $Y'^F$.
The same argument shows that $Y'^F$ dominates $Y^F$ and we conclude that they are equivalent.
\end{proof}

We write $R_{L/F}(Y)$ and $R_{L'/F}(Y')$ for the $F$-varieties given by the Weil transfers of the $L$-variety $Y$ and the $L'$-variety $Y'$ respectively.
Since the groups $G$ and $G'$ are $p'$-inner and $p$-consistent, the conditions of Lemma~\ref{R} apply to $Y$ and $Y'$. Therefore,
under the conditions of Theorem~\ref{crit}, and using the notations introduced in the proof, we get:
\begin{cor}
\label{critR}
The motives $U_A(Y)$ and $U_{A'}(Y')$ are isomorphic if and  only if
the Artin motives $A$ and $A'$ are isomorphic and the Weil transfers $R_{L/F}(Y)$ and  $R_{L'/F}(Y')$ are equivalent.
\qed
\end{cor}

From now on, all the motives considered belong to the full additive subcategory of $\CM(F,\F)$ generated by direct summands of Tate shifts of motives of geometrically split
varieties satisfying the nilpotence principle. The Krull–Schmidt property holds for all objects in this category
by~\cite[Corollary 35]{MR2264459} and ~\cite[Corollary 2.6]{upper}
(see also~\cite[Proposition 2.1]{TateTraces}), so we may give the following definition:
\begin{dfn}
\label{ATTrace.dfn}
The \emph{Artin-Tate trace} of a motive $M$ is the part in a complete decomposition of $M$ which consists of Artin-Tate motives.
We say that two motives have the same \emph{higher Artin-Tate trace} if
over all field extensions of $F$
their Artin-Tate traces are isomorphic.
\end{dfn}
We get the following classification theorem, which extends~\cite[Theorem 4.3]{TateTraces}.
\begin{thm}
\label{isomcrit}
Let $M$ and $M'$ be $F$-motives. We assume that each summand in a complete decomposition of any of them is isomorphic to a Tate shift of an A-upper motive of a $p'$-inner and $p$-consistent reductive group. The motives $M$ and $M'$ are isomorphic if and only if they have the same higher Artin-Tate trace.
\end{thm}
\begin{rem}
By Theorem~\ref{main}, this applies to direct sums of shifts of motives of projective homogeneous varieties under $p'$-inner and $p$-consistent reductive groups.
\end{rem}
\begin{proof}[Proof of Theorem \ref{isomcrit}]
If $M$ and $M'$ are isomorphic, then by the Krull-Schmidt property, they have the same higher Artin-Tate trace.

Conversely, assume that $M$ and $M'$ have the same higher Artin-Tate trace. We prove that $M$ and $M'$ are isomorphic by induction on the maximum of the number of summands in their respective complete motivic decompositions.
If this maximum is zero, both $M$ and $M'$ are trivial.
If it is nonzero, write
$$
M=U_{A_1}(X_1)\{n_1\}\oplus..\oplus U_{A_k}(X_k)\{n_k\}\mbox{  and  } M'=U_{B_1}(Y_1)\{m_1\}\oplus ...\oplus U_{B_s}(Y_s)\{m_s\},
$$
where $A_i$ and $B_j$ are Artin motives and $X_i$ and $Y_j$ are corestrictions of projective homogeneous varieties defined over some separable field extensions of $F$.
We may assume that $n=\min_{1\leq i\leq k} n_i$ is not higher than $m=\min_{1\leq j \leq s}m_j$.
Pick an integer $1\leq \alpha \leq k$ such that the Weil transfer $R(X_{\alpha})$ to $F$ of $X_\alpha$
is minimal for the domination relation among the $R(X_i)$'s such that $U_{A_i}(X_i)\{n\}$ is a direct summand in the above decomposition of $M$ (to lighten notation, we write here $R(\cdot)$ for Weil transfers, dismissing the associated finite separable extensions).

Over the function field of $R(X_\alpha)$ the motive $M$ contains as a summand the
Artin-Tate motive $A\{n\}$ with $A:=(A_\alpha)_{F(R(X_{\alpha}))}$.
By assumption on the higher Artin-Tate traces of $M$ and $M'$,
it follows that
the motive $M'$ over the same function field also contains $A\{n\}$.
Moreover, $A$ is indecomposable by Corollary~\ref{AFX}.
So $A\{n\}$ is a summand in $U_{B_\beta}(Y_{\beta})\{m_\beta\}_{F(R(X_\alpha))}$ for some $1\leq \beta\leq s$.
Hence, we have $m_\beta=m=n$ and $U(Y_\beta^F)_{F(R(X_\alpha))}$ contains a Tate summand $\F$, see Remark~\ref{Charles}.
The variety $R(X_{\alpha})$ dominates $R(Y_{\beta})$, and $(B_\beta)_{F(R(X_{\alpha}))}$ is isomorphic to $A$.
Applying again Corollary~\ref{AFX}, we get that the Artin motives $A_{\alpha}$ and $B_{\beta}$ are isomorphic.

The same argument over the function field of $R(Y_{\beta})$
yields
some $1\leq \gamma \leq k$ such that $R(X_{\gamma})$ is dominated by $R(Y_{\beta})$,
$A_\gamma\simeq A_\alpha$, and $n_\gamma=n$.

By minimality of $R(X_{\alpha})$, the varieties $R(X_{\alpha})$, $R(X_\gamma)$ and $R(Y_{\beta})$ are equivalent.
The A-upper motives $U_A(X_{\alpha})$ and $U_A(Y_{\beta})$ are then isomorphic by Corollary \ref{critR}.
Induction, applied to the complementary summands in $M$ and $M'$ of $U_A(X_{\alpha})\{n\}$ and $U_A(Y_{\beta})\{n\}$, proves that $M$ and $M'$ are isomorphic.
%
\end{proof}

\begin{cor}\label{pp'bridge}
The motives of two projective homogeneous varieties for two absolutely simple groups of type different from $^3\!\cat{D}_4$ and $^6\!\cat{D}_4$ are isomorphic if and only if they have the same higher Artin-Tate trace.
\end{cor}

\begin{proof}
Let $G$, $G'$ be two absolutely simple groups of (possibly different) type not  $^3\!\cat{D}_4$, nor $^6\!\cat{D}_4$. Assume that $X$ is projective $G$-homogeneous and $Y$ is projective $G'$-homogeneous.

If the motives of $X$ and $Y$ have the same higher Artin-Tate trace, then they clearly share the same higher Tate trace as well. The case where both $G$ and $G'$ are $p$-inner then boils down to \cite[Theorem 4.3]{TateTraces}.

If $G$ and $G'$ are both $p'$-inner and $p$-consistent, Theorem \ref{main} asserts that both motives $M(X)$ and $M(Y)$ can be written as direct sums of Tate shifts of A-upper motives. We thus land in the conditions of Theorem \ref{isomcrit}.

In the remaining case, one of the group, say $G$, is $p$-inner, while the other, $G'$, is $p'$-inner and $p$-consistent. For any field extension $E/F$, the Artin-Tate trace of $M(Y_E)$ becomes pure Tate over any prime-to-$p$ field extension over which $G'$ is of inner type. Since the Tate trace of $M(X_E)$ is invariant over such an extension by \cite[Lemma 5.9]{TateTraces}, it follows that if the higher Artin-Tate traces of $M(X)$ and $M(Y)$ are isomorphic, they actually correspond to the higher Tate traces of $M(X)$ and $M(Y)$. Assume an A-upper motive $U_A(Z)\{i\}\simeq U(\hat{Z})\otimes A\{i\}$ is a direct summand of $M(Y)$. Then $M(Y_{F(\hat{Z})})$ contains an indecomposable direct summand isomorphic to $A_{F(\hat{Z})}$ (see ~\eqref{Charles.eq} and Corollary \ref{AFX}) which must be a Tate motive by the previous discussion. It follows that $A$ is itself a Tate motive, and $M(Y)$ can be written as a direct sum of Tate shifts of \emph{upper motives}. The motive $M(Y)$ thus fulfills the conditions of \cite[Theorem 4.3]{TateTraces}, as well as $M(X)$.
\end{proof}

\begin{rem}
    Corollary \ref{pp'bridge} is stated for motives of projective homogeneous varieties, but holds more generally for their arbitrary direct summands (with the same proof).
\end{rem}
\section{Motivic equivalence for reductive groups}
\label{Motivic equivalence}

Motivic equivalence for algebraic groups has been introduced by the first author in~\cite{motequiv}.
Roughly speaking two inner reductive groups with the same Dynkin diagram are called motivic equivalent if their respective projective homogeneous varieties of any given type have isomorphic motives. This notion can be extended to non-inner reductive groups using corestriction of motives.
The main result of this section is Corollary~\ref{moteq.cor}, which provides a criterion of motivic equivalence for $p'$-inner $p$-consistent reductive algebraic groups which are inner forms of each other, in terms of their higher Tits $p$-indexes.
Combining this with a similar result for $p$-inner reductive groups proved in~\cite{TateTraces}, we get that the criterion actually holds for all absolutely simple algebraic groups of type other than $^3\!\cat{D}_4$ and $^6\!\cat{D}_4$.

We start with a proposition which provides conditions under which we may extend scalars to a function field to detect isomorphim for some A-upper motives.
\begin{prop}
\label{free}
Let $X$ be a projective homogeneous $F$-variety.
 Consider two reductive algebraic groups $G$ and $G'$ over $F$.
Let  $L/F$ and $L'/F$ be two $p'$-extensions (i.e. finite separable field extensions with Galois closure of degree prime to $p$), and pick some indecomposable summands $A$ in $M(L)^F$ and $A'$ in $M(L')^F$.
Let $Y$, $Y'$ be  projective homogeneous varieties over $L$ and $L'$ under
$G_L$ and $G'_{L'}$,  respectively, which are equivalent to the restrictions of some projective homogeneous $F$-varieties.
If $Y^F$ and $Y'^F$ both dominate $X$, and if the A-upper $F(X)$-motives $U_{A_{F(X)}}(Y_{L(X)})$ and $U_{A'_{F(X)}}(Y'_{L'(X)})$ are isomorphic, then the $F$-motives $U_A(Y)$ and $U_{A'}(Y')$ are isomorphic as well.
\end{prop}

\begin{rem} By Corollary~\ref{AFX}, $A_{F(X)}$ is an indecomposable summand of $M\bigl(L(X)\bigr)^{F(X)}$, and similarly for $A'_{F(X)}$. So the A-upper motives considered in the statement of Proposition~\ref{free} are well defined, see Notation~\ref{A-upper.not}. \end{rem}

\begin{proof}[Proof of Proposition \ref{free}]
Assume that $U_{A_{F(X)}}(Y_{L(X)})$ and $U_{A'_{F(X)}}(Y'_{L'(X)})$ are isomorphic.
By Theorem \ref{crit}, the $F(X)$-motives $A_{F(X)}$ and $A'_{F(X)}$ are isomorphic and the $F(X)$-varieties $(Y^F)_{F(X)}$ and $(Y'^F)_{F(X)}$ are equivalent.
Corollary~\ref{AFX} shows that $A\simeq A'$, and by~\cite[Proof of Proposition 9]{motequiv}, $Y^F$ and $Y'^F$ are equivalent.
The conclusion follows applying again Theorem~\ref{crit}.
\end{proof}

Given a reductive group $G$ over $F$, we
denote by $D_G$ its Dynkin diagram, that is the Dynkin diagram of the root system of $G_{\bar F}$ with respect to $T_{\bar F}$, for some maximal torus $T\subset G$.
Sometimes, depending on the context, $D_G$ stands for the set of vertices of the Dynkin diagram.
We also denote by $E$ a minimal field extension of $F$ over which $G$ becomes inner, so that the action of $\Gamma_F$ on $D_G$ (also called the $\ast$-action) factors through $\Gal(E/F)$.

Any $\ast$-invariant subset $\tau\subset D_G$ yields a projective $G$-homogeneous variety denoted by $X_{G,\tau}$, and this induces a
 bijection between the $\ast$-invariant subsets of $D_G$ and
the isomorphism classes of projective $G$-homogeneous varieties. Note that we use the same convention as in the proof of Theorem~\ref{main} for this identification; in particular, the empty set corresponds to $\Spec F$.

A vertex of $D_G$ is called distinguished if it is contained in an orbit $\tau$ such that the variety $X_{G,\tau}$ has a rational point.
The classical Tits index of $G$ consists of its Dynkin diagram $D_G$ , endowed with the $\ast$-action, together with the subset $D_G^0\subset D_G$ which consists of all distinguished vertices.
In a similar way, a vertex is called $p$-distinguished if it is contained in an orbit $\tau$ such that the variety $X_{G,\tau}$ is isotropic (mod $p$), that is admits a closed point of degree prime to $p$. We denote by $D_G^p$ the set of all $p$-distinguished vertices, see~\cite{titspindexes}.

In the proofs, we will use the fact that a projective homogeneous variety $X$ is isotropic (mod $p$) if and only if its upper motive is a Tate motive, see~\cite[Lemma 2.2]{TateTraces}. Combining with~\cite[Corollary 2.15]{upper}, we get that two equivalent varieties are isotropic over the same extensions of their base field.

Consider now two reductive groups $G$ and $G'$. We assume each of them is an inner form of the other,
or, equivalently, both are inner forms of the same quasi-split group.
In such a situation, the Dynkin diagrams $D_G$ and $D_{G'}$ are
$\Gamma_F$-equivariant isomorphic and we will fix one of the possible isomorphisms.

\begin{prop}
\label{uppertits}
Let $G$ and $G'$ be $p'$-inner $p$-consistent reductive groups over $F$, inner forms of each other. Fix an equivariant isomorphism $\varphi\colon D_G\rightarrow D_{G'}$ of their Dynkin diagrams, and a $\ast$-invariant subset $\tau_0$ of $D_G$. Let $E/F$ be a minimal field extension over which $G_E$, hence also $G'_E$, is of inner type.
The following conditions on $G$, $G'$, $\tau_0$, and $\phi$ are equivalent:
    \begin{enumerate}
        \item[(i)] for any field extension $K/F$, one has $\tau_0 \subset D^p_{G_K}$ (i.e., $\tau_0$ is $p$-distinguished over $K$) if and only if $\varphi(\tau_0) \subset D^p_{G'_K}$;
         moreover,  for any $K/F$ over which these equivalent conditions hold, we have $\varphi(D^{p}_{G_K})=D^{p}_{G'_K}$;
        \item[(ii)] for any field extension $L/F$ contained in $E$, any indecomposable summand $A$ of the motive $M(L)^F$, and any $\Gal(E/L)$-invariant subset $\tau\subset D_G$ containing $\tau_0$, the A-upper motives $U_A(X_{G_L,\tau})$ and $U_A(X_{G'_L,\varphi(\tau)})$ are isomorphic.
    \end{enumerate}
\end{prop}

\begin{proof}
   Assuming (i), fix a field extension $L/F$ contained in $E$, an Artin motive $A$, and a subset $\tau\supset\tau_0$ as in (ii).
    The subset $\tau_0$ is $p$-distinguished for $G$ over the function field $\tilde{L}$ of the variety $X_{G_L,\tau}$.
    Therefore, by condition (i), the subset $\varphi(\tau)\subset D_{G'}$ is $p$-distinguished over $\tilde{L}$.
    The $L$-variety $X_{G_L,\tau}$ thus dominates $X_{G'_L,\varphi(\tau)}$.
    The same reasoning with $\varphi(\tau)$ and the inverse of $\varphi$ implies that the $L$-varieties $X_{G_L,\tau}$ and $X_{G'_L,\varphi(\tau)}$ are equivalent.
Hence, the $F$-varieties $X^F_{G_L,\tau}$ and $X^F_{G'_L,\varphi(\tau)}$ are equivalent (see Remark~\ref{cores.rem}), and the A-upper motives
    $U_A(X_{G_L,\tau})$ and $U_A(X_{G'_L,\varphi(\tau)})$ are isomorphic by Theorem \ref{crit}.\\

   Let us now prove the converse. Condition (ii) applied to $L=F$ and $\tau=\tau_0$ shows that the upper motives $U(X_{G,\tau_0})$ and $U(X_{G',\varphi(\tau_0)})$ are isomorphic.
    Therefore, given a field extension $K/F$, the variety $X_{G_K,\tau_0}$ is isotropic if and only if $X_{G'_K,\varphi(\tau_0)}$ is isotropic as well. This means that over the field $K$, $\tau_0$ is $p$-distinguished for $G$ if and only if $\phi(\tau_0)$ is for $G'$.

    To prove the second part of (i), let us fix a field extension $K/F$ such that $\tau_0$ is $p$-distinguished over $K$, and denote $\Theta=D_{G_K}^p\subset D_G$.
    By definition of $\Theta$, the $K$-variety $X_{G_K,\Theta}$ is isotropic, so there exists a prime-to-$p$ field extension $M/K$ such that $X_{G_M,\Theta}$ has a rational point. Let $L/F$ be a minimal subfield of $E$ such that $\Theta$ is $\Gal(E/L)$ invariant.
    By~\cite[Theorem 3.16(i)]{kerstenrehmann}, replacing $E$ by an isomorphic field extension of $F$ if necessary, we may assume $L\subset M$.
Condition (ii) applied to $L$ and $\tau=\Theta$ now provides an isomorphism of the A-upper motives $U_A(X_{G_L,\Theta})$ and $U_A(X_{G'_L,\varphi(\Theta)})$.
Hence, the $L$-varieties $X_{G_L,\Theta}$ and $X_{G'_L,\varphi(\Theta)}$ are equivalent (mod $p$).
As $L$ is contained in $M$, it follows that $X_{G_M,\Theta}$ and $X_{G'_M,\varphi(\Theta)}$ also are equivalent.
The first one has a rational point, so the second is isotropic; since $[M:K]$ is prime to $p$, it follows that $X_{G'_K,\varphi(\Theta)}$ also is isotropic.

Thus, the subset $\varphi(\Theta)=\varphi(D^{p}_{G_K})$ is $p$-distinguished for $G'$ over $K$.
    The same reasoning with $G$ replaced by $G'$, $\tau_0$ by $\varphi(\tau_0)$, and $\varphi$ by its inverse, gives that $\varphi^{-1}(D^p_{G'_K})\subset D^p_{G_K}$.
    Hence $\varphi(D^{p}_{G_K})=D^{p}_{G'_K}$.
\end{proof}

Consider an arbitrary subset $\tau$ of $D_G$.
There exists a minimal field extension $L_{\tau}/F$ contained in $E/F$ such that $\tau$ is $\Gal(E/L_\tau)$-invariant.
The $F$-motive $M_{G,\tau}:=M(X_{G_{L_{\tau}},\tau})^F$ is called the \emph{standard motive of $G$ of type $\tau$}.

If $\tau$ is $\ast$-invariant, it is simply the motive of the projective $G$-homogeneous variety $X_{G,\tau}$.

\begin{thm}\label{equiv}
Let $G$ and $G'$ be $p'$-inner $p$-consistent reductive groups over a field $F$ which are inner forms of each other. Let $\tau_0$ be an invariant subset in $D_G$.
The equivalent conditions of Proposition \ref{uppertits} are satisfied if and only if for any subset $\tau\subset D_G$ containing $\tau_0$, the motives
$M_{G,\tau}$ and $M_{G',\varphi(\tau)}$
are isomorphic.
\end{thm}

\begin{proof}
The ``if'' part is clear: if the motives $M_{G,\tau}$ and $M_{G',\varphi(\tau)}$ are isomorphic, then
for any field $L$ with $L_\tau\subset L\subset E$,
the varieties $X_{G_{L},\tau}^F$ and $X_{G'_{L},\varphi(\tau)}^F$ are equivalent.
Hence, by Theorem \ref{crit}, $G$ and $G'$ satisfy condition $(ii)$ of Proposition \ref{uppertits}.

We prove the opposite implication by induction on the (common) semisimple rank of $G$ and $G'$.
More concretely, assuming the conditions of Proposition \ref{uppertits}, we will prove
that for every $\tau\supset\tau_0$ the motives $M_{G,\tau}$ and $M_{G',\varphi(\tau)}$
are isomorphic.
For $\tau=\emptyset$ the isomorphism trivially holds.
This covers the rank zero case, which is the base of the induction.
Below we assume that $\tau\ne\emptyset$.

We first show
that $M_{G,\tau}$ and $M_{G',\varphi(\tau)}$ are isomorphic if $\tau$ and $\varphi(\tau)$ are $\Gal(E/F)$-invariant and the associated varieties both have a rational point (hence the reductive algebraic groups $G$ and $G'$ are isotropic).

Let $\tilde{G}$ be the semisimple part of a parabolic subgroup in $G$ of type $\tau$. The Dynkin diagram $D_{\tilde{G}}$ of $\tilde{G}$ is obtained by removing the subset $\tau$ from $D_G$, and $\tilde{G}_{E}$ is of inner type.
By \cite[Theorem 4.2]{MR2178658}, there is a motivic decomposition
$$
M_{G,\tau}\simeq \bigoplus_{i\in \mathcal{I}} M_{\tilde{G}_{\!L_i},\tau_i}^F\{n_i\}
$$
with some field extensions $L_i/F$ contained in $E$ and some $\Gal(E/L_i)$-invariant $\tau_i\subset D_{\tilde{G}}$.
Note that the fields $L_i$, the projective $\tilde{G}_{L_i}$-homogeneous varieties $X_{\tilde{G}_{\!L_i},\tau_i}$, and the shifting numbers $n_i$ in this decomposition are completely determined by the underlying combinatorics of $G$. In particular, the isomorphism $\varphi:D_G\rightarrow D_{G'}$ from Proposition \ref{uppertits} yields an analogous decomposition of $M_{G',\varphi(\tau)}$ with respect to the semisimple part $\tilde{G'}$ of a parabolic subgroup in $G'$ of type $\phi(\tau)$,
with the same $\mathcal{I}$, $L_i$, $\tau_i$, and $n_i$:
$$
M_{G',\varphi(\tau)}\simeq \bigoplus_{i\in \mathcal{I}} M_{\tilde{G'}_{\!\!L_i},\varphi(\tau_i)}^F\{n_i\}
$$
Since $G$ and $G'$ are inner forms of each other and satisfy condition (i) of Proposition \ref{uppertits} for $\tau_0\subset D_G$, the groups
$\tilde{G}$ and $\tilde{G'}$ also do, for $\tilde\tau_0=\emptyset$.
Indeed, since $X_{G,\tau}$ and $X_{G',\varphi(\tau)}$ are isotropic, for any field extension $K/F$, we have disjoint union decompositions
$$D^p_{G_K}=D^p_{\tilde{G}_K}\sqcup \tau\;\;\text{  and  }\;\;D^p_{G'_K}=D^p_{\tilde{G'}_{\! K}}\sqcup \phi(\tau).$$
Condition (i) of Proposition \ref{uppertits} for $G$ and $G'$ gives that $D^p_{G'_K}=\varphi(D^p_{G_K})$ and hence  $D^p_{\tilde{G'}_{\!K}}=\varphi(D^p_{\tilde{G}_{\! K}})$.
It follows that for any any $i\in \mathcal{I}$ and any field extension $L_i/F$, the reductive groups $\tilde{G}_{\!L_i}$ and $\tilde{G'}_{\!\!L_i}$ also satisfy condition (i) of Proposition \ref{uppertits} with respect to the restriction of $\varphi$ and the subset $\tilde{\tau}_0=\emptyset$. By induction, the motives $M_{\tilde{G}_{\!L_i},\tau_i}$ and $M_{\tilde{G'}_{\!\!L_i},\varphi(\tau_i)}$ are thus isomorphic.
Therefore, the motives $M_{\tilde{G}_{\!L_i},\tau_i}^F$ and $M_{\tilde{G'}_{\!\!L_i},\varphi(\tau_i)}^F$ are isomorphic as well
and so $M_{G,\tau}\simeq M_{G',\varphi(\tau)}$.

The second case we consider is that of arbitrary $\Gal(E/F)$-invariant subsets $\tau$ and $\varphi(\tau)$.
We will reduce to the previous case using scalar extension to the function field of a variety as in~Proposition~\ref{free}.
Let us first introduce a set of integers describing motivic decompositions.
For any projective $G$-homogeneous variety $X$ and any direct summand $M$ in $M(X)$, we write $\ell_{A,Y,n}(M)$ for the number of indecomposable summands isomorphic to the A-upper summand $U_A(Y)\{n\}$ in a complete decomposition of $M$, where $U_A(Y)$ is a A-upper motive of $G$, see Definition~\ref{AupperG} and Theorem~\ref{main}.
Since $G$ and $G'$ satisfy the conditions of Proposition~\ref{uppertits}, the A-upper motives of $G$ and $G'$ are pairwise isomorphic, and we may also consider the number of indecomposable summands isomorphic to $U_A(Y)$ in a direct summand of the motive of a $G'$-projective homogeneous variety, for which we use a similar notation.
If the motives of $X_{G,\tau}$ and $X_{G',\varphi(\tau)}$ are not isomorphic, then $\ell_{A,Y,n}(M_{G,\tau})\neq \ell_{A,Y,n}(M_{G',\varphi(\tau)})$ for some A-upper motive $U_A(Y)$ of $G$. Consider the minimal integer $n$ for which such a non-equality occurs.

Over the function field $K/F$ of the product $\Pi:=X_{G,\tau}\times X_{G',\varphi(\tau)}$ both $X_{G,\tau}$ and $X_{G',\varphi(\tau)}$ have a rational point. So $M_{G_K,\tau}$ and $M_{G'_K,\varphi(\tau)}$ are isomorphic.
Moreover, the motive $A_K$ is indecomposable (see Corollary \ref{AFX})
and so we can investigate the integer $\ell_{A_K,Y_K,n}(M_{G_K,\tau})$.
To lighten notation (by abusing it), below we will write $Y_K$ for the variety $Y_{L(\Pi)}$.
If $U_{A_K}(Y_K)\{n\}$ is a direct summand of $M_{G_K,\tau}$,
then by the Krull-Schmidt property and Theorem \ref{main},
it is a direct summand in the $K/F$-restriction $(U_B(Z)\{k\})_K$ of an A-upper motive $U_B(Z)$ of $G$, shifted by some integer $k$.
By~\eqref{Charles.eq}, we have $(U_B(Z))_K\simeq U_{B_K}(Z_K)\oplus N,$ where $N$ is a direct sum of A-upper motives with Tate shifts at least $1$; therefore, $k\leq n$.
Since $X_{G,\tau}$ and $X_{G',\varphi(\tau)}$ are equivalent, any projective homogeneous variety which dominates $X_{G,\tau}$ (or $X_{G',\varphi(\tau)}$) dominates their product. In particular, Proposition \ref{free} implies that a direct summand $U_{A_K}(Y_K)\{n\}$ of $M(X_{G_K,\tau})$ may only arise from a $K/F$-restriction $(U_B(Z)\{n\})_K$ with the same shift $n$ if $B$ and $A$ are isomorphic Artin motives and $Z$ and $Y$ are equivalent, that is from the A-upper motive $U_A(Y)\{n\}$ (see Theorem \ref{crit}).

Write $M$ and $M'$ for the direct summands of $M_{G,\tau}$ and $M_{G',\varphi(\tau)}$, respectively, given by the sum of all the indecomposable summands isomorphic to A-upper motives of $G$ which shifts strictly lower than $n$ (in a fixed complete decomposition).
Thanks to the previous discussion, separating the summands $U_{A_K}(Y_K)\{n\}$ of $M_{G_K,\tau}$ which arise from $M_K$,
we get the equalities
$$
\ell_{A_K,Y_K,n}(M_{G_K,\tau})=\ell_{A,Y,n}(M_{G,\tau})+\ell_{A_K,Y_K,n}(M_K),
$$
$$
\mbox{ and }\ell_{A_K,Y_K,n}(M_{G'_K,\varphi(\tau)})=\ell_{A,Y,n}(M_{G',\varphi(\tau)})+\ell_{A_K,Y_K,n}(M'_K).
$$
By minimality of $n$, the motives $M$ and $M'$ are isomorphic, hence $M_K$ and $M'_K$ are isomorphic as well and $\ell_{A_K,Y_K,n}(M_K)=\ell_{A_K,Y_K,n}(M'_K)$.
As by assumption $\ell_{A,Y,n}(M_{G,\tau})\neq \ell_{A,Y,n}(M_{G',\varphi(\tau)})$, it follows that $\ell_{A_K,Y_K,n}(M_{G_K,\tau})$ and $\ell_{A_K,Y_K,n}(M_{G'_K,\varphi(\tau)})$ are not equal, a contradiction to the fact that the motives of $X_{G_K,\tau}$ and of $X_{G'_K,\varphi(\tau)}$ are isomorphic (recall that both of these varieties have a rational point).

Finally, consider an arbitrary subset $\tau$ of $D_G$.
The reductive groups $G_{L_\tau}$ and $G'_{L_{\tau}}$ satisfy condition (i) of Proposition \ref{uppertits}.
It follows from the Galois-invariant case that
the motives
$M_{G_{F_{\tau}},\tau}$ and $M_{G'_{F_{\tau}},\varphi(\tau)}$ are isomorphic, hence so are
the motives
$M_{G,\tau}=M_{G_{F_{\tau}},\tau}^F$ and $M_{G',\varphi(\tau)}=M_{G'_{F_{\tau}},\varphi(\tau)}^F$, and this finishes the proof.
\end{proof}

A field is called \emph{$p$-special} if every finite extension of this field has a $p$-power degree.
Let $G$ and $G'$ be two reductive groups, inner forms of each other.
Similarly to \cite[Definition 1]{motequiv},
we say  that
$G$ and $G'$ are \emph{motivic equivalent} (with coefficients in $\mathbb{F}$) with respect to a Galois-equivariant isomorphism $\varphi:D_G\rightarrow D_{G'}$,  if for any subset $\tau$ of $D_G$, the motives $M_{\tau,G}$ and $M_{\varphi(\tau),G'}$ are isomorphic.

\begin{cor}
\label{moteq.cor}
Let $G$ and $G'$ be $p'$-inner $p$-consistent reductive algebraic groups over $F$, inner forms of each other.
Let $\phi$
be a $\ast$-equivariant isomorphism of their Dynkin diagrams.
The groups $G$ and $G'$ are motivic equivalent with respect to $\varphi$ if and only if for any $p$-special field extension $K/F$, $\varphi$ identifies the Tits indexes of $G_K$ and $G'_K$.
\end{cor}

\begin{proof}
Theorem \ref{equiv} with $\tau_0=\emptyset$ states that $G$ and $G'$ are motivic equivalent with respect to $\varphi$ if and only if for any field extension $K/F$, $\varphi$ identifies the subsets of $p$-distinguished vertices of $G_K$ and $G'_K$.
 Over a $p$-special field $K$,
 this expresses as $\varphi(D^0_{G_K})=D^0_{G'_K}$ (through classical Tits indexes), proving one implication.
 The converse also holds since a variety is isotropic if and only if it has a rational point over a $p$-special closure of its base field \cite[Proof of Lemma 4.11]{TateTraces}.
\end{proof}

\bibliographystyle{acm}
\bibliography{atum}

\def\cprime{$'$} \def\cftil#1{\ifmmode\setbox7\hbox{$\accent"5E#1$}\else \setbox7\hbox{\accent"5E#1}\penalty 10000\relax\fi\raise 1\ht7 \hbox{\lower1.15ex\hbox to 1\wd7{\hss\accent"7E\hss}}\penalty 10000 \hskip-1\wd7\penalty 10000\box7}
\begin{thebibliography}{10}

\bibitem{MR2178658}
{\sc Brosnan, P.}
\newblock On motivic decompositions arising from the method of {B}ia\l ynicki-{B}irula.
\newblock {\em Invent. Math. 161}, 1 (2005), 91--111.

\bibitem{MR2110630}
{\sc Chernousov, V., Gille, S., and Merkurjev, A.}
\newblock Motivic decomposition of isotropic projective homogeneous varieties.
\newblock {\em Duke Math. J. 126}, 1 (2005), 137--159.

\bibitem{MR2264459}
{\sc Chernousov, V., and Merkurjev, A.}
\newblock Motivic decomposition of projective homogeneous varieties and the {K}rull-{S}chmidt theorem.
\newblock {\em Transform. Groups 11}, 3 (2006), 371--386.

\bibitem{motequiv}
{\sc De~Clercq, C.}
\newblock Motivic equivalence of semisimple algebraic groups.
\newblock {\em Compos. Math. 153}, 10 (2017), 2195--2213.

\bibitem{titspindexes}
{\sc De~Clercq, C., and Garibaldi, S.}
\newblock Tits {{\(p\)}}-indexes of semisimple algebraic groups.
\newblock {\em J. Lond. Math. Soc., II. Ser. 95}, 2 (2017), 567--585.

\bibitem{TateTraces}
{\sc De~Clercq, C., and Qu\'eguiner-Mathieu, A.}
\newblock Higher {T}ate traces of {C}how motives.
\newblock {\em J. Reine Angew. Math. 815\/} (2024), 1--22.

\bibitem{EKM}
{\sc Elman, R., Karpenko, N., and Merkurjev, A.}
\newblock {\em The algebraic and geometric theory of quadratic forms}, vol.~56 of {\em American Mathematical Society Colloquium Publications}.
\newblock American Mathematical Society, Providence, RI, 2008.

\bibitem{shells}
{\sc Garibaldi, S., Petrov, V., and Semenov, N.}
\newblock Shells of twisted flag varieties and the {R}ost invariant.
\newblock {\em Duke Math. J. 165}, 2 (2016), 285--339.

\bibitem{MR1388895}
{\sc Hovey, M., Palmieri, J.~H., and Strickland, N.~P.}
\newblock Axiomatic stable homotopy theory.
\newblock {\em Mem. Amer. Math. Soc. 128}, 610 (1997), x+114.

\bibitem{MR1809664}
{\sc Karpenko, N.~A.}
\newblock Weil transfer of algebraic cycles.
\newblock {\em Indag. Math. (N.S.) 11}, 1 (2000), 73--86.

\bibitem{MR1992018}
{\sc Karpenko, N.~A.}
\newblock On the first {W}itt index of quadratic forms.
\newblock {\em Invent. Math. 153}, 2 (2003), 455--462.

\bibitem{outer}
{\sc Karpenko, N.~A.}
\newblock Upper motives of outer algebraic groups.
\newblock In {\em Quadratic forms, linear algebraic groups, and cohomology}, vol.~18 of {\em Dev. Math.} Springer, New York, 2010, pp.~249--258.

\bibitem{MR3022954}
{\sc Karpenko, N.~A.}
\newblock Isotropy of orthogonal involutions.
\newblock {\em Amer. J. Math. 135}, 1 (2013), 1--15.
\newblock With an appendix by Jean-Pierre Tignol.

\bibitem{upper}
{\sc Karpenko, N.~A.}
\newblock Upper motives of algebraic groups and incompressibility of {S}everi-{B}rauer varieties.
\newblock {\em J. Reine Angew. Math. 677\/} (2013), 179--198.

\bibitem{kerstenrehmann}
{\sc Kersten, I., and Rehmann, U.}
\newblock Generic splitting of reductive groups.
\newblock {\em T{\^o}hoku Math. J. (2) 46}, 1 (1994), 35--70.

\bibitem{MR1632779}
{\sc Knus, M.-A., Merkurjev, A., Rost, M., and Tignol, J.-P.}
\newblock {\em The book of involutions}, vol.~44 of {\em American Mathematical Society Colloquium Publications}.
\newblock American Mathematical Society, Providence, RI, 1998.
\newblock With a preface in French by J.\ Tits.

\bibitem{MR3821516}
{\sc Linckelmann, M.}
\newblock {\em The block theory of finite group algebras. {V}ol. {I}}, vol.~91 of {\em London Mathematical Society Student Texts}.
\newblock Cambridge University Press, Cambridge, 2018.

\bibitem{MR0258836}
{\sc Manin, J.~I.}
\newblock Correspondences, motifs and monoidal transformations.
\newblock {\em Mat. Sb. (N.S.) 77 (119)\/} (1968), 475--507.

\bibitem{MR1321819}
{\sc Scheiderer, C.}
\newblock {\em Real and \'etale cohomology}, vol.~1588 of {\em Lecture Notes in Mathematics}.
\newblock Springer-Verlag, Berlin, 1994.

\bibitem{MR0224710}
{\sc Tits, J.}
\newblock Classification of algebraic semisimple groups.
\newblock In {\em Algebraic {G}roups and {D}iscontinuous {S}ubgroups ({P}roc. {S}ympos. {P}ure {M}ath., {B}oulder, {C}olo., 1965)}. Amer. Math. Soc., Providence, R.I., 1966, pp.~33--62.

\bibitem{MR2066515}
{\sc Vishik, A.}
\newblock Motives of quadrics with applications to the theory of quadratic forms.
\newblock In {\em Geometric methods in the algebraic theory of quadratic forms}, vol.~1835 of {\em Lecture Notes in Math.} Springer, Berlin, 2004, pp.~25--101.

\bibitem{Vishik-u-invariant}
{\sc Vishik, A.}
\newblock Fields of {$u$}-invariant {$2^r+1$}.
\newblock In {\em Algebra, arithmetic, and geometry: in honor of {Y}u. {I}. {M}anin. {V}ol. {II}}, vol.~270 of {\em Progr. Math.} Birkh\"auser Boston Inc., Boston, MA, 2009, pp.~661--685.

\bibitem{MR2031199}
{\sc Voevodsky, V.}
\newblock Motivic cohomology with {${\bf Z}/2$}-coefficients.
\newblock {\em Publ. Math. Inst. Hautes \'{E}tudes Sci.}, 98 (2003), 59--104.

\bibitem{MR3642472}
{\sc Wildeshaus, J.}
\newblock Notes on {A}rtin-{T}ate motives.
\newblock In {\em Autour des motifs---\'{E}cole d'\'{e}t\'{e} {F}ranco-{A}siatique de {G}\'{e}om\'{e}trie {A}lg\'{e}brique et de {T}h\'{e}orie des {N}ombres/{A}sian-{F}rench {S}ummer {S}chool on {A}lgebraic {G}eometry and {N}umber {T}heory. {V}ol. {III}}, vol.~49 of {\em Panor. Synth\`eses}. Soc. Math. France, Paris, 2016, pp.~101--131.

\end{thebibliography}

\end{document}